\documentclass[10pt]{article}

\usepackage{vmargin}

\setpapersize{A4}
\setmargins{3.8cm}       
{1.5cm}                        
{13.5cm}                      
{23.42cm}                    
{10pt}                           
{1cm}                           
{0pt}                             
{2cm}

\usepackage{hyperref}
\usepackage{amsfonts}
\usepackage{amsmath}
\usepackage{amsthm}
\usepackage{amssymb}
\usepackage[all]{xy}
\usepackage[latin1]{inputenc}
\usepackage{graphicx}
\usepackage{color}
\usepackage{yfonts}
\input xy
\xyoption{all}
\usepackage{graphicx}
\usepackage{mathrsfs}
\usepackage{float}
\usepackage{framed}
\usepackage{titlesec}

\newtheorem{theorem}{Theorem}[section]
\newtheorem*{theo}{Theorem}
\newtheorem{corollary}[theorem]{Corollary}
\newtheorem{lemma}[theorem]{Lemma}
\newtheorem{proposition}[theorem]{Proposition}
\theoremstyle{definition}
\newtheorem{definition}[theorem]{Definition}
\newtheorem{remark}[theorem]{Remark}

\newcommand*{\longhookrightarrow}{\ensuremath{\lhook\joinrel\relbar\joinrel\rightarrow}}

\newcommand{\slas}{/\!\!/ }

\newcommand{\Ox}{\mathscr{Oo}_{X}}
\newcommand{\N}{\mathscr{N}}

\newcommand{\K}{\mathscr{K}}

\newcommand{\Oo}{\mathscr{O}}

\titleformat{\section}[display]{\scshape \bfseries}{\S \thesection\filcenter}{1ex}{\fillast}

\titleformat{\subsection}[hang]{\itshape\bfseries}{\thesubsection. --- }{0pt}{\upshape}

\titleformat{\subsubsection}[runin]{\itshape\filleft}{\thesubsubsection.---}{0pt}{}

\newcommand{\eq}[1][r]
   {\ar@<-3pt>@{-}[#1]
    \ar@<-1pt>@{}[#1]|<{}="gauche"
    \ar@<+0pt>@{}[#1]|-{}="milieu"
    \ar@<+1pt>@{}[#1]|>{}="droite"
    \ar@/^2pt/@{-}"gauche";"milieu"
    \ar@/_2pt/@{-}"milieu";"droite"}

\date{}
\title{
\textbf{A compactification of the universal moduli space of principal G-bundles}
\author{
\textsc{\'Angel Luis Mu\~noz Casta\~neda}
\footnote{Departamento of Matem\'aticas, Universidad de Le\'on, Spain, email: {\texttt amunc@unileon.es}}
}
}

\begin{document}

\maketitle

\begin{abstract}
Given a semisimple linear algebraic group $G$ and a natural number $g\geq 2$, we prove the existence of a compactification of the universal moduli space of semistable principal $G$-bundles over $\overline{\textrm{M}}_{g}$, the fibers of which over singular curves are the moduli spaces of $\delta$-semistable singular principal $G$-bundles for $\delta$ a natural number large enough.
\end{abstract}

\vspace{1cm}

{\textsl {Keywords:}} {\small principal bundles, swamps, universal moduli, stable curves}

{\textsl {2010 MSC: 14H60; 14D20.} }

\tableofcontents

\section{Introduction}

Let $X$ be a smooth projective curve of genus $g\geq 2$ over an algebraically closed field, $\mathbb{C}$, of characteristic $0$ and $G$ a connected reductive algebraic group. 
In \cite{Rama,Rama2}, A. Ramanathan proved the existence of a projective moduli space, $\textrm{M}^{(s)s}_{X}(G)$, of (semi)stable principal $G$-bundles on $X$ with fixed topological type. When $G=\textrm{GL}_{n}$, this moduli space is isomorphic to the classical moduli space of (semi)stable vector bundles constructed by D. Mumford and C. S. Seshadri  \cite{mumford0,Seshadri2}. 

A natural problem is to consider a degeneration of $X$ along with a discrete valuation ring $A$, whose special fiber is a stable curve of genus $g$, and to describe what the limit of a semistable principal $G$-bundle on $X$ looks like when we approach the special fiber. When $G=\textrm{GL}_{n}$ (resp. $\textrm{Sp}_{n}$, $\textrm{O}_{n}$), this problem is equivalent to the problem of finding the limit of a locally free sheaf of rank $n$ (resp. together with a bilinear form satisfying certain conditions), and is well-known that the solution is a torsion-free sheaf (resp. together with a bilinear form satisfying certain conditions) \cite{Fal}. The case $G=\textrm{SL}_{n}$ has been studied in \cite{sun1,sun2} by X. Sun. For a general group $G$, V. Balaji studies flat degenerations of the moduli space when the special fiber of the relative curve is irreducible \cite{Balaji}.

One can go further, and study how the moduli space of semistable principal $G$-bundle varies when the base curve moves along with the moduli of space of stable curves, $\overline{\textrm{M}}_{g}$. This leads to the problem of compactifying the universal moduli space of principal $G$-bundles over $\overline{\textrm{M}}_{g}$. Regarding this problem, not so much is known except for the case $G=\textrm{GL}_{n}$ \cite{Pando}.

An alternative construction of the moduli space of principal $G$-bundles was given by A. Schmitt for $G$ a semisimple linear algebraic group. Given a faithful representation $\rho:G\hookrightarrow \textrm{SL}(V)$ of dimension $r$, a singular principal $G$-bundle is a pair $( \mathscr{E},\tau)$ consisting on a vector bundle $ \mathscr{E}$ of rank $r$ and degree $0$ together with a non-trivial morphism of algebras $\tau:S^{\bullet}(V\otimes \mathscr{E})^{G}\rightarrow \mathscr{O}_{X}$. Giving $\tau$ is the same as giving a morphism $X\rightarrow\underline{\textrm{Hom}}_{\mathscr{O}_{X}}(V\otimes\mathscr{O}_{X}, \mathscr{E}^{\vee})\slas G$ and the singular principal $G$-bundle $( \mathscr{E},\tau)$ is said to be honest if $\tau$ takes values in the subscheme of local isomorphisms $\underline{\textrm{Isom}}_{\mathscr{O}_{X}}(V\otimes\mathscr{O}_{X}, \mathscr{E}^{\vee})\slash G$. Schmitt's work is based on the following result (a particular case of \cite[Proposition 9]{serre}),
\begin{equation}\label{equiv}
\left\{
\begin{array}{l}  
\textrm{isomorphism classes of} \\ 
\textrm{principal }G \textrm{-bundles on X}
\end{array}  
\right\} \simeq 
\left\{  \begin{array}{l} 
\textrm{isomorphism classes of pairs }(\mathscr{E},\tau) \\ 
\textrm{where } \mathscr{E}  \textrm{ is a locally free sheaf of}\\ 
\textrm{rank }r\textrm{ with trivial determinant} \\ 
\textrm{and }\tau\colon X\rightarrow \textrm{\textrm{\underline{Isom}}}_{\Ox}(V\otimes \Ox,  \mathscr{E})/G  
\end{array} 
\right\}.
\end{equation}
In \cite{Alexander4}, it is proved the existence of a projective moduli space, $\textrm{SPB}(\rho)^{\delta\textrm{-}(s)s}_{X,P}$, of $\delta$-semistable singular principal $G$-bundles with Hilbert polynomial $P$ on a smooth projective curve, for a given $\delta\in\mathbb{Q}_{>0}$ and $\rho:G\hookrightarrow \textrm{SL}(V)$. Furthermore, it is proved that for $\delta\gg0$, every $\delta$-semistable singular principal $G$-bundle is honest and that $\textrm{SPB}(\rho)^{\delta\textrm{-}(s)s}_{X,P}$ is isomorphic to $\textrm{M}_{X}^{(s)s}(G)$. Taking the above into account, A. Schmitt's approach seems to be suitable to handle the problems mentioned above for an arbitrary reductive group G.

The goal of this work is to prove the existence of a compactification of the moduli problem defined by pairs $(X,P)$, where $X$ is a smooth projective curve of genus $g$ and $P$ is a semistable principal $G$-bundle. The approach we will follow is A. Schmitt's approach so, thanks to Equation (\ref{equiv}), we will be able to substitute $P$ by its corresponding singular principal $G$-bundle. A key ingredient in the theory of singular principal bundles is the notion of swamp, that is, a pair $(\mathscr{F},\phi)$ where $\mathscr{F}$ is a coherent sheaf of pure dimension one and $\phi:(\mathscr{F}^{\otimes a})^{\oplus b}\rightarrow \mathscr{O}_{X}$ is a morphism of $\mathscr{O}_{X}$-modules \cite{L}. From \cite{AMC}, it follows that the construction of the moduli space for singular principal bundles is essentially reduced to the construction of the moduli space of swamps, so a great part of this work is focused on the universal moduli space of these objects.

The precise statement of the main result is the following.
\begin{theo}[Theorem \ref{Mainteo2}, Section \ref{fibers}]
Let $G$ be a semisimple linear algebraic group, $\rho:G\hookrightarrow \emph{SL}(V)$ a faithful representation, $V$ being a vector space of dimension $r$, $g\geq 2$ a natural number, $\delta\in\mathbb{Q}_{>0}$ a rational number, and $P(n)=(2g-2)rn+r(1-g)\in\mathbb{Z}[n]$ a polynomial of degree one.
There exists a projective scheme $\emph{SPB}(\rho)^{\delta\emph{-(s)s}}_{g,P}$ together with a map 
$$\Theta_{sb}:\emph{SPB}(\rho)^{\delta\emph{-(s)s}}_{g,P}\rightarrow\overline{\emph{M}}_{g}$$
satisfying that for any stable curve $[X]\in\overline{\emph{M}}_{g}$, 
$\Theta_{sb}^{-1}([X])=\emph{SPB}(\rho)^{\delta\emph{-(s)s}}_{X,P}/\emph{Aut}(X).$
Making $\delta$ large enough, we have
$\Theta_{sb}^{-1}([X])=\textrm{M}^{(s)s}_{X}(G)/\emph{Aut}(X)$
for every smooth curve, $[X]\in\overline{\emph{M}}_{g}$. Thus, the universal moduli space of semistable principal $G$-bundles $M_{g}^{(s)s}(G)$ is an open subscheme of the projective scheme $\emph{SPB}(\rho)^{\delta\emph{-(s)s}}_{g,P}$.
\end{theo}

The main technical result for proving the above theorem is the following.
 \begin{theo}[Theorem \ref{bound3}]
Let $g,h,C\in\mathbb{N}$ with $g\geq 2$ and $\underline{P}$ a finite set of polynomials of degree one with integral coefficients. There is a natural number $N_{0}$ depending only on $\underline{P}, C,g,h$ such that for every Cohen-Macaulay projective and connected curve of genus $g$ with a very ample line bundle $\mathscr{O}_{X}(1)$ of degree $h$ and every coherent sheave of pure dimension one over $X$, $\mathscr{F}$, with Hilbert polynomial in $\underline{P}$ that satisfies $\mu_{\textrm{max}}(\mathscr{F})\leq C$, we have $h^{1}(X,\mathscr{F}(k))=0$ for all $k\geq N_{0}$ and $\mathscr{F}(k)$ is generated by its global sections.
 \end{theo}

This result is necessary to prove that there exists a relative Quot scheme over the parameter space of the universal curve of genus $g$ containing every coherent sheaf of pure dimension one with fixed rank and degree appearing in a $\delta$-semistable singular principal bundle. 
Note that, in case $\underline{P}$ consists of one polynomial $P(n)=\alpha n+ \beta$ and $C=\beta/\alpha$, this shows that there is a Quot scheme relative to the parameter space for stable curves of genus $g$ in which every semistable sheaf with Hilbert polynomial $P$ appears (see \cite[Section 5, Section 6]{Pando}).

Following the proof of \cite[Theorem 3.6]{sols} and applying the corollaries of Theorem \ref{bound3} appropriately we can prove the following result. 

 \begin{theo}[Theorem \ref{teoremon}]
Let $g,h,D,a,b\in\mathbb{N}$ with $g\geq 2$ and $P(n)\in\mathbb{Z}[n]$ a polynomial of degree one. There are natural numbers $N,L\in\mathbb{N}$ depending only on the numerical input data such that for every $k\geq N$ and every $l\geq L$ the following holds: for every Cohen-Macaulay projective and connected curve of genus $g$ with a very ample line bundle $\mathscr{O}_{X}(1)$ of degree $h$, a point $(q,\Phi)\in Z_{P,\mathscr{D},a,b}^{k,l}(X)$ is $\emph{GIT}$-(semi)stable with respect to $\mathscr{O}_{Z_{P,\mathscr{D},a,b}^{k,l}(X)}(n_{1},n_{2})$ if and only if the corresponding swamp $( \mathscr{F},\phi)$ is $\delta$-(semi)stable and the linear map $f_{q}\colon F\rightarrow H^{0}(X, \mathscr{F}(k))$ is an isomorphism, where $F:=\mathbb{C}^{P(k)}$.
 \end{theo}
Here $Z_{P,\mathscr{D},a,b}^{k,l}(X)$ is the parameter space for swamps, and $\mathscr{O}_{Z_{P,\mathscr{D},a,b}^{k,l}(X)}(n_{1},n_{2})$ is a certain very ample line bundle on it. From \cite{AMC}, it follows the corresponding result for singular principal bundles which, in turn, implies that the projective scheme we construct is a coarse moduli space for pairs $(X,(\mathscr{F},\tau))$ given by a  stable curve and a $\delta$-semistable singular principal bundle.

\subsection{The strategy and known results} 

The necessary steps for the construction of a compactification of the universal moduli space of principal $G$-bundles by adding singular principal bundles on stable curves are summarized below. Although the construction of the moduli space for a single curve is known, we include them in the summary for the sake of clarity of the exposition.

\subsubsection{The fiber-wise problem} Let $X$ be a stable curve of genus $g$. The moduli space $\textrm{SPB}(\rho)^{\delta\textrm{-}(s)s}_{X,P}$ of $\delta$-semistable singular principal $G$-bundles with Hilbert polynomial $P$ is constructed in several steps. 
\begin{enumerate}
\item It is shown that there are natural numbers, $a$ and $b$, large enough such that to every singular principal $G$-bundle $\tau:S^{\bullet}(V\otimes \mathscr{F})^{G}\rightarrow \mathscr{O}_{X}$ with Hilbert polynomial $P$ we can assign a swamp $\phi_{\tau}:((V\otimes \mathscr{F})^{\otimes a})^{\oplus b}\rightarrow\mathscr{O}_{X}$ with Hilbert polynomial $P$, and this mapping, $\tau\mapsto\phi_{\tau}$, is injective. Then, a singular principal $G$-bundle $( \mathscr{F},\tau)$ is said to be $\delta$-semistable if the corresponding swamp $( \mathscr{F},\phi_{\tau})$ is $\delta$-semistable. Thus, the construction of the moduli space of singular principal $G$-bundles is essentially reduced to the construction of the moduli space of swamps. 
\item It is shown that there is a natural number $N$ large enough such that for every $\delta$-semistable swamp, $( \mathscr{F},\phi)$, and for every $k\geq N$, $ \mathscr{F}(k)$ is generated by global sections and $h^{1}(X, \mathscr{F}(k))=0$. This implies that every $\delta$-semistable swamp $( \mathscr{F},\phi)$ defines a point, $[(q,\Phi)]$, in the projective scheme 
\begin{equation*}
\textrm{Quot}^{P}_{\mathbb{C}^{P(k)}\otimes\mathscr{O}_{X}(-k)/X/\mathbb{C}}\times \mathbf{P}((((\mathbb{C}^{P(k)})^{\otimes a})^{\oplus b})^{\vee}\otimes H^{0}(X,\mathscr{O}_{X}(ak)))
\end{equation*}
where $q:\mathbb{C}^{P(k)}\otimes\mathscr{O}_{X}(-k)\twoheadrightarrow  \mathscr{F}$ and $\mathbb{C}^{P(k)}\simeq H^{0}(X, \mathscr{F}(k))$ is a fixed isomorphism.
\item It is shown that, given $k\geq N$, there is a natural number $L$ large enough such that for every $l\geq L$, the closed immersion 
$$i_{k,l}:\textrm{Quot}^{P}_{\mathbb{C}^{P(k)}\mathscr{O}_{X}(-k)/X/\mathbb{C}}\hookrightarrow \mathbf{P}(\bigwedge^{P(l)}(\mathbb{C}^{P(k)}\otimes H^{0}(X,\mathscr{O}_{X}(l-k))))$$ 
defined by the Grothendieck embedding of the Quot scheme composed with Pl\"ucker embedding satisfies the following property:  a point $(i_{k,l}\times id)([( \mathscr{F},\phi)])$ in the projective scheme
\begin{equation}\label{space2}
\mathbf{P}(\bigwedge^{P(l)}(\mathbb{C}^{P(k)}\otimes H^{0}(X,\mathscr{O}_{X}(l-k))))  \times \mathbf{P}(((V^{\otimes a})^{\oplus b})^{\vee}\otimes H^{0}(X,\mathscr{O}_{X}(ak)))
\end{equation}
is GIT semistable for the action of $\textrm{SL}_{P(k)}$ and respect to certain polarization $\mathscr{L}$ 
if and only if $( \mathscr{F},\phi)$ is $\delta$-semistable and $\mathbb{C}^{P(k)}\simeq H^{0}(X, \mathscr{F}(k))$. 
\item The construction of both moduli spaces, $\textrm{SPB}(\rho)^{\delta\textrm{-}(s)s}_{P,X}$ and $\mathcal{T}_{P,X,a,b}^{\delta\textrm{-}(s)s}$, is as follows. It is shown that there exist closed subschemes $\textrm{Z}_{bund}(X)\subset \textrm{Z}_{sw}(X)$ of the projective scheme given in Equation (\ref{space2}) parametrizing (rigidified) singular principal $G$-bundles and swamps respectively. The action of $\textrm{SL}_{P(k)}$ on this scheme induces an action of both schemes, $\textrm{Z}_{bund}(X)$ and $\textrm{Z}_{sw}(X)$. From Step 3, we conclude that $\textrm{Z}^{ss}_{sw}(X)\slas\textrm{SL}_{P(k)}$ exists and is projective, which in turn implies that $\textrm{Z}_{bund}^{ss}(X)\slas\textrm{SL}_{P(k)}$ exists and is projective as well. Again, from Step 3, we conclude that $\textrm{Z}^{ss}_{sw}(X)\slas\textrm{SL}_{P(k)}$ and $\textrm{Z}_{bund}^{ss}(X)\slas\textrm{SL}_{P(k)}$ are coarse projective moduli spaces for $\delta$-semistable swamps and $\delta$-semistable singular principal $G$-bundles. 
\end{enumerate}

\subsubsection{The relative problem}\label{relative}
Fix integers $g\geq 2$, $d=10(2g-2)$ and $M=d-g$. Consider the Hilbert functor of projective curves in $\mathbf{P}^{M}$ of genus $g$ and degree $d$.
Let us denote by $\textrm{H}_{g,d,M}$ the representative of $\textbf{Hilb}_{d,g}$ and $\textrm{H}_{g}\subset \textrm{H}_{g,d,M}$ the locus of non-degenerate, 10-canonical stable curves of genus $g$. The scheme $\textrm{H}_{g}$ is a locally closed subscheme and it is a nonsingular, irreducible quasi-projective variety.

Fix an isomorphism for each stable curve, $X$, of genus $g$, $\mathbb{C}^{M+1}\simeq H^{0}(X,\mathscr{\omega}_{X}^{10})$, and let $h(s)=ds-g+1$ be a polynomial in $s$ of degree $1$. 
Recall that there exists an integer $s_{0}$ such that $ \forall s\geq s_{0}$,
$i'_{s}:\textrm{H}_{d,g,N}\hookrightarrow \textbf{Grass}(h(s),H^{0}(\mathbf{P}^{M},\mathscr{O}_{\mathbf{P}^{M}}(s))^{\vee})$
is a closed immersion. Moreover, there exists $s_1$ such that the GIT linearized problem $i'_{s_1}$ satisfies:
(1) $\textrm{H}_{g}$ belongs to the semistable locus,
(2) $\textrm{H}_{g}$ is closed in the semistable locus.
Finally, the action of $\textrm{SL}_{M+1}$ on $\mathbf{P}^{M}$ induces an action on $\textrm{H}_{g}$, and Gieseker shows that
$
\overline{\textrm{M}}_{g}=\textrm{H}_{g}/\textrm{SL}_{M+1}.
$
The scheme $\textrm{H}_{g}$ is endowed with a universal family $U_{g}\rightarrow \textrm{H}_{g}$ called the universal curve.

\begin{enumerate}
\item It is shown that all the numbers appearing in the fiber-wise problem (that need to be large enough) do not depend on the base curve. That is, there are $a,b,N,L$ that works for every stable curve of genus $g$.
\item Relative parameter spaces  $\textrm{Z}_{bund,g}\rightarrow H_{g}$ and $\textrm{Z}_{sw,g}\rightarrow H_g$ for families of $\delta$-semistable singular principal $G$-bundles and families of $\delta$-semistable swamps on the fibers of $U_{g}\rightarrow H_{g}$ are constructed. Again $\textrm{Z}_{bund,g}\subset \textrm{Z}_{sw,g}$.
\item $\textrm{Z}_{bund,g}$ and $\textrm{Z}_{sw,g}$ are embedded in the projective scheme $\mathbf{P}(H_1)\times\mathbf{P}(H_2)\times\mathbf{P}(H_3)$, where
\begin{equation}
\begin{split}
H_{1}&=(((V)^{\otimes a})^{\oplus b})^{\vee}\otimes H^{0}(\mathbf{P}^{M},\mathscr{O}_{\mathbf{P}^{M}}(ak))\\
H_{2}&=\bigwedge^{h(s)}(H^{0}(\mathbf{P}^{M},\mathscr{O}_{\mathbf{P}^{M}}(s)))\\
H_{3}&=\bigwedge^{P(l)}(\mathbb{C}^{P(k)}\otimes H^{0}(\mathbf{P}^{M},\mathscr{O}_{\mathbf{P}^{M}}(l-k)))\\
\end{split}
\end{equation}
\item The natural action of $\textrm{SL}_{P(k)}\times\textrm{SL}_{M+1}$ on $\mathbf{P}(H_1)\times\mathbf{P}(H_2)\times\mathbf{P}(H_3)$ induces an action on $\textrm{Z}_{bund,g}$ and $\textrm{Z}_{sw,g}$. Then, it is shown that $\textrm{Z}^{ss}_{sw,g}\slas (\textrm{SL}_{P(k)}\times\textrm{SL}_{M+1})$ exists and is projective, which in turn implies that $\textrm{Z}_{bund,g}^{ss}\slas (\textrm{SL}_{P(k)}\times\textrm{SL}_{M+1})$ exists and is projective as well. Finally, Step 1 implies that $\textrm{Z}_{bund,g}^{ss}\slas (\textrm{SL}_{P(k)}\times\textrm{SL}_{M+1})$ coarsely represents the moduli functor for pairs, and that there is $\delta\in\mathbb{Q}_{>0}$ large enough such that the fibers over nonsingular curves are precisely the moduli spaces of semistable principal $G$-bundles over them (quoted out by the group of automorphisms of the base curve).
\end{enumerate}
\subsubsection{Known results}
Regarding the fiber-wise problem, it is worth noting that the moduli space of $\delta$-semistable swamps has been constructed in \cite{L} over any projective scheme of pure dimension one, even in the relative case, while the existence of a projective moduli space of $\delta$-semistable singular principal $G$-bundles over any stable curve (and also in the relative case) has been proved in \cite{AMC}. 
The problem regarding the uniform behavior along with $\overline{\textrm{M}}_{g}$ of the natural numbers $a,b$ of Paragraph 1.1.1. 1) has been already solved in \cite{AMC} as well.
However, the problem of the uniform behavior of the numerical parameters $N,L$ of Paragraphs 1.1.1. 2) and 3) remains unsolved, and it is essential for the construction of the universal compactification. Regarding the relative problem, an important fact is that the actions of the groups $\textrm{SL}_{n}$ and $\textrm{SL}_{M+1}$ commute with each other, which makes the relative GIT problem easier to handle (see \cite{Pando} for the case of vector bundles). This fact allows us to make use of some technical results proved in \cite{Pando} to show the existence of a projective moduli space for $\delta$-semistable swamps over $\overline{\textrm{M}}_{g}$.

\subsection{Outline of the paper}

This paper is organized as follows. In Section \ref{sectionuniform} we deal we Step $1$ of Paragraph \ref{relative}, and the main result is Theorem \ref{bound3}. In Section \ref{comparision}, we apply it to the key steps of the construction of the moduli space of $\delta$-semistable swamps, and we show the uniform behavior along with $\overline{\textrm{M}}_{g}$ of the numerical parameters involved in such construction. We end up with the existence of a projective moduli space of $\delta$-semistable swamps of given Hilbert polynomial over a stable curve of genus $g$. This result has been proved in \cite{L}, but it has been included for the sake of clarity of the exposition.

In Section \ref{sectionuniversal}, we prove the existence of a coarse projective moduli space for the moduli functor defined by pairs $(X,( \mathscr{F},\phi))$, $X$ being a stable curve of genus $g$ and $( \mathscr{F},\phi)$ a $\delta$-semistable swamp of given Hilbert polynomial. The forgetful map defines a morphism between this moduli space and $\overline{\textrm{M}}_{g}$.

Finally, in Section \ref{sectionmain} we show the existence of a coarse projective moduli space for the moduli functor defined by pairs $(X,( \mathscr{F},\tau))$, $X$ being a stable curve of genus $g$ and $( \mathscr{F},\tau)$ a $\delta$-semistable singular principal $G$-bundle of given rank and degree $0$. As in the case of swamps, there is a morphism to $\overline{\textrm{M}}_{g}$. This, together with the results given in  \cite{AMC, Alexander2,Alexander1, Asch} and Section \ref{sectionuniform}, implies that the fibers of the above morphism over nonsingular curves are precisely the classical moduli spaces of principal $G$-bundles constructed by A. Ramanathan, and they form an open subset of the constructed moduli space.

The base field will be taken to be an algebraically closed field of characteristic $0$. The word curve will mean a $\mathbb{C}$-scheme of finite type and pure dimension one.


\section{A uniform boundedness result on Cohen-Macaulay curves}
\label{sectionuniform}

Although the aim is to construct moduli spaces over stable curves, the results of this section hold for Cohen-Macaulay curves, so the results are stated for this more general case.

\subsection{Preliminaries}

We will state some notation and recall some basic facts. We expose them specialized to the case of curves. 

Let $X$ be a Cohen-Macaulay projective and connected curve of genus $g$ together with a very ample invertible sheaf, $\mathscr{O}_{X}(1)$.

\subsubsection{Regularity}
Let $k\in\mathbb{Z}$ be an integer. A coherent sheaf, $\mathscr{F}$, on $X$, is $k$-regular if 
$H^{1}(X,\mathscr{F}(k-1))=0$.
If $\mathscr{F}$ is $k$-regular and $k'>k$, then $\mathscr{F}$ is also $k'$-regular. From Serre's Vanishing theorem, it follows that there is always an integer $k$ such that $\mathscr{F}$ is $k$-regular. The regularity of $\mathscr{F}$ is defined by
$$\textrm{reg}(\mathscr{F})=\textrm{inf}\{k\in\mathbb{Z}:\mathscr{F} \textrm{ is }k\textrm{-regular}\}.$$

Given a family of equivalence classes of coherent sheaves whose Hilbert polynomials belong to a finite set in $\mathbb{Z}[n]$, we can decide whether the family is bounded or not by looking at the regularity of the members of the family (see (\cite[Lemma 1.7.6]{Huyb}).

\subsubsection{Polarized slope}
Given a coherent sheaf, $\mathscr{F}$, on $X$ its (polarized) degree and its (polarized) slope are defined as 
\begin{equation}
\begin{split}
\textrm{deg}(\mathscr{F}):=\chi(\mathscr{F})-r\chi(\mathscr{O}_{X}), \
\mu(\mathscr{F}):=\dfrac{\chi(\mathscr{F})}{\alpha},
\end{split}
\end{equation}
$\alpha$ being the multiplicity of $\mathscr{F}$ (the degree one coefficient of its Hilbert polynomial) and $r=\alpha/h$ its rank. 
Observe that the degree of a coherent sheaf $\mathscr{F}$ is determined by its Hilbert polynomial, the degree of the very ample line bundle we have fixed, $\mathscr{O}_{X}(1)$, and the genus $g$ of the curve.

\subsubsection{Pure sheaves}
A coherent sheaf on $X$ is of pure dimension one if $\textrm{dim}(\textrm{Supp}(\mathscr{G}))=1$ for every $\mathscr{G}\subseteq\mathscr{F}$.  Recall that a coherent sheaf of pure dimension one, $\mathscr{F}$, is semistable if for any subsheaf $\mathscr{F}'\subset \mathscr{F}$,
$\mu(\mathscr{F}')\leq\mu(\mathscr{F}).$
Recall also that for any coherent sheaf of pure dimension one, $\mathscr{F}$, there is a unique filtration (Harder-Narasimhan filtration)
\begin{equation*}
0=\mathscr{F}_{0}\subset\mathscr{F}_{1}\subset\hdots\subset\mathscr{F}_{k}=\mathscr{F}
\end{equation*}
such that the quotients $\mathscr{F}_{i}/\mathscr{F}_{i-1}$ are semistable sheaves with decreasing slopes.  As usual, we will use the following notation,
\begin{align*}
\mu_{max}(\mathscr{F})&=\textrm{max}\{\mu(\mathscr{F}_{i}/\mathscr{F}_{i-1})|i=1,\hdots,k\}=\mu(\mathscr{F}_{1}), \\ 
\mu_{min}(\mathscr{F})&=\textrm{min}\{\mu(\mathscr{F}_{i}/\mathscr{F}_{i-1})|i=1,\hdots,k\}=\mu(\mathscr{F}/\mathscr{F}_{k-1}).
\end{align*}
For any subsheaf $\mathscr{G}\subset\mathscr{F}$, we have $\mu(\mathscr{G})\leq\mu_{\textrm{max}}(\mathscr{F})$.


\subsection{Vector bundles on the projective line}

Let $\mathscr{E}$ be a vector bundle of rank $r$ on the projective line $\mathbf{P}^{1}_{k}$. By \cite{groth0}, we know that there are integers $n_{1}\geq\hdots\geq n_{r}$ such that 
$$\mathscr{E}\simeq\bigoplus_{i=1}^{r}\mathscr{O}(n_{i}).$$
The tuple $(n_{1},\hdots,n_{r})$ is defined as the type of $\mathscr{E}$ and it is denoted by $\tau(\mathscr{E})$. We will denote by $\tau_{min}(\mathscr{E})$ (respectively, $\tau_{max}(\mathscr{E})$) the minimum (respectively, maximum) integer of the type $\tau(\mathscr{E})$ of $\mathscr{E}$, that is, $\tau_{\textrm{min}}(\mathscr{E})=n_{r}$ (respectively $\tau_{\textrm{max}}(\mathscr{E})=n_{1}$).

Let $r,m\in\mathbb{N}$, and $d\in\mathbb{Z}$ be integers. There are finitely many isomorphism classes of locally free sheaves on $\mathbf{P}^{1}$ of rank $r$, degree $d$, and such that $h^{0}(\mathbf{P}^{1},\mathscr{E})=m$. Such isomorphism classes are determined by the tuples of integers  $n_{1}\geq\hdots\geq n_{r}$ verifying the equations
\begin{equation}\label{tuple}
\begin{split}
\delta(n_{1})+\hdots+\delta(n_{r})&=m,  \textrm{ where }\delta(n_{i})=\left\{ \begin{array}{cc} n_{i}+1 & \textrm{ if }n_{i}\geq 0\\ 0 & \textrm{ otherwise}  \end{array} 
\right.\\
n_{1}+\hdots+n_{r}&=d
\end{split}
\end{equation}
Let us denote by $N(r,d,m)$ the set of tuples of integers $n_{1}\geq\hdots\geq n_{r}$ satisfying Equation (\ref{tuple}) and by $S_{-}(r,d,m)$ (resp. $S_{+}(r,d,m)$) the minimum (resp. maximum) among the integers $i\in\mathbb{Z}$ that appear as the smaller (resp. largest) integer in a tuple of $N(r,d,m)$. Therefore, any integer $n\in \mathbb{Z}$ that appears in a tuple $(n_{1},\hdots,n_{r})\in N(r,d,m)$ satisfies that $S_{+}(r,d,m)\geq n\geq S_{-}(r,d,m)$.

  \begin{lemma}\label{bound00}
Let $\mathscr{E}$ be a locally free sheaf of rank $r$, degree $d$ and $h^{0}(\mathbf{P}^{1},\mathscr{E})=m$. Then, for every $n\geq -S_{-}(r,d,m)$, $\mathscr{E}(n)$ is generated by its global sections and $h^{1}(\mathbf{P}^{1}, \mathscr{E}(n))=0$.
 \end{lemma}
\begin{proof}
Let $\tau( \mathscr{E})=(n_{1},\hdots,n_{r})$ be the type of $\mathscr{E}$, that is, $\mathscr{E}=\bigoplus_{i=1}^{r}\mathscr{O}_{X}(n_{i})$. Then, $\mathscr{E}(n)$ is generated by global sections if and only if $n+n_{i}\geq 0$ for all $i$, that is, if and only if $n\geq -n_{i}$ for all $i$. However, $n_{i}\geq S_{-}(r,d,m)$ for every $i$ by definition, so taking $n\geq -S_{-}(r,d,m)$ we get the desired result. On the other hand $h^{1}(\mathbf{P}^{1}, \mathscr{E}(n))=\sum_{i=1}^{r}h^{0}(\mathbf{P}^{1},\mathscr{O}_{\mathbf{P}^{1}}(-2-n_{i}-n))$, and $h^{0}(\mathbf{P}^{1},\mathscr{O}_{\mathbf{P}^{1}}(-2-n_{i}-n))=0$ for every $n\geq-S_{-}(r,d,m)$, so we are done.
\end{proof}

\subsection{Uniform boundedness}

Let $X$ be a Cohen-Macauly projective and connected curve together with a very ample line bundle $\mathscr{O}_{X}(1)$ of degree $h$. 
By \cite[Proposition 6]{mum-red} there exists a finite surjective morphism $f:X\rightarrow \mathbf{P}^{1}$ such that $f^{*}\mathscr{O}_{\mathbf{P}^{1}}(1)\simeq \mathscr{O}_{X}(1)$ and \cite[Theorem 23.1]{matsumura} implies that $f$ is a flat. 
Let  $\mathscr{L}$ be a coherent sheaf over $X$ of degree $d$ and multiplicity $\alpha$, and consider $\mathscr{E}:=f_{*}\mathscr{L}$.
Since $f$ is finite and $f^{*}\mathscr{O}_{\mathbf{P}^{1}}(1)\simeq \mathscr{O}_{X}(1)$, we know that $P_{\mathscr{E}}(n)=P_{\mathscr{L}}(n)$. Assuming $n\gg0$, we get
\begin{equation}\label{num}
\begin{split}
P_{\mathscr{E}}(n)&=\textrm{rk}(\mathscr{E})\cdot n+\textrm{rk}(\mathscr{E})(1-g)+\textrm{deg}(\mathscr{E}),\\
P_{\mathscr{L}}(n)&=\alpha\cdot n+\dfrac{\alpha}{h}(1-g)+d.
\end{split}
\end{equation}
hence, $\textrm{rk}(\mathscr{E})=\alpha$ and $\textrm{deg}(\mathscr{E})=(1-g)(\dfrac{\alpha}{h}-\alpha)+d$. 

Assume now that $\mathscr{L}$ is an invertible sheaf. Then, $\mathscr{E}$ is locally free and there are integers $a_{1}(f)\geq\hdots \geq a_{h}(f)$ such that
\begin{align*}
\mathscr{E}:=f_{*}\mathscr{L}=\bigoplus_{i=1}^{h}\mathscr{O}_{\mathbf{P}^{1}_{k}}(a_{i}(f)), \ (1-g)(1-h)+d=\sum_{i=1}^{h}a_{i}(f).
\end{align*}

 \begin{definition}
Let $X,\mathscr{O}_{X}(1), f,\mathscr{L}$ be as above. We define the $f$-type of $\mathscr{L}$ as the tuple $(a_{1}(f),\cdots, a_{h}(f))$, and it is denoted by $\tau_{f}(\mathscr{L})$.
 \end{definition}

  \begin{lemma}\label{uniform1}
Let $g,h,m\in\mathbb{N}$ and $d\in\mathbb{Z}$. There are integers $S_{-}, S_{+}$ depending only on $g,h,m,d$ such that for any Cohen-Macaulay projective and connected curve of genus $g$ with a very ample line bundle $\mathscr{O}_{X}(1)$ of degree $h$, any line bundle $\mathscr{L}$ on $X$ of degree $d$ and $h^{0}(X,\mathscr{L})=m$, and any finite morphism $f:X\rightarrow\mathbf{P}^{1}$ such that $f^{*}\mathscr{O}(1)\simeq\mathscr{O}_{X}(1)$, the following holds:
\begin{equation}\label{ineq0}
S_{+}\geq \tau_{f,max}( \mathscr{E}), \ \tau_{f,min}(\mathscr{E})\geq S_{-},  \textrm{ where }\mathscr{E}:=f_{*}(\mathscr{L}).
\end{equation}
Furthermore, for every $n\geq -S_{-}$, $\mathscr{L}(n)$  is generated by its global sections and $h^{1}(X,\mathscr{L}(n))=0$.
 \end{lemma}
\begin{proof}
By Equation (\ref{num}), $\textrm{rk}(\mathscr{E})=h$ and $\textrm{deg}(\mathscr{E})=(1-g)(1-h)+d$. On the other hand $h^{0}(\mathbf{P}^{1},\mathscr{E})=h^{0}(X,\mathscr{L})=m$. Then, $S_{-}:=S_{-}(h,(1-g)(1-h)+d,m)$ and $S_{+}:=S_{+}(h,(1-g)(1-h)+d,m)$ satisfy the inequality given in Equation (\ref{ineq0}). Besides, if $n\geq -S_{-}$, $ \mathscr{E}(n)$ is generated by global sections and $h^{1}(\mathbf{P}^{1}, \mathscr{E}(n))=0$ by Lemma \ref{bound00}. This implies that $h^{1}(X,\mathscr{L}(n))=0$, since $h^{1}(X,\mathscr{L}(n))=h^{1}(\mathbf{P}^{1}, \mathscr{E}(n))$, and that $\mathscr{L}(n)$ is generated by its global sections, since the adjunction map, $f^{*}f_{*}\rightarrow \textrm{id}$, is surjective for finite morphisms.
\end{proof}


 \begin{lemma}\label{bound2-3}
Let $g,h\in\mathbb{N}$ and  $n\in\mathbb{Z}$. There exists a constant $C\in\mathbb{Z}$, depending only on $g,h,n$, such that for every Cohen-Macaulay projective and connected curve $X$ of genus $g$ with a very ample line bundle $\mathscr{O}_{X}(1)$ of degree $h$, we have $\mu_{\textrm{max}}(\mathscr{H})\leq C$, where $\mathscr{H}=\oplus_{i=1}^{r}\mathscr{O}_{X}(n_i)$ with $r$ any natural number and $n_i\leq n$.
 \end{lemma}
\begin{proof}
Let $X$ and $\mathscr{H}$ be as in the statement.
Let $f:X\rightarrow\mathbf{P}^{1}$ be a finite morphism such that $f^{*}\mathscr{O}_{\mathbf{P}^{1}}(1)\simeq\mathscr{O}_{X}(1)$. Then, $f_{*}\mathscr{H}=\oplus_{i=1}^{r} \mathscr{E}(n_i)$, where $ \mathscr{E}=f_{*}\mathscr{O}_{X}=\bigoplus_{i=1}^{h}\mathscr{O}_{\mathbf{P}^{1}_{k}}(a_{i}(f))$. By Equation (\ref{num}), we know that $\textrm{rk}(\mathscr{E})=h$ and $\textrm{deg}(\mathscr{E})=(1-g)(1-h)$. Let $ \mathscr{F}\subset\mathscr{H}$ be a  subsheaf of multiplicity $\alpha$. Since $f_{*} \mathscr{F}$ is locally free, there are integers $t_1,\hdots,t_\alpha$ such that $f_{*} \mathscr{F}=\oplus_{i=1}^{\alpha}\mathscr{O}_{X}(t_i)$. Let $n'$ be the maximum among  $n_1,\hdots,n_r$. Since $f_{*} \mathscr{F}\subset f_{*}\mathscr{H}$, we deduce that $t_{i} \leq \tau_{\textrm{max}}( \mathscr{E})+n'\leq \tau_{\textrm{max}}( \mathscr{E})+n$, which implies that $\textrm{deg}(f_{*} \mathscr{F})\leq \alpha(\tau_{\textrm{max}}( \mathscr{E})+n)$. On the other hand, $\textrm{deg}( \mathscr{F})=\alpha(1-g)+\textrm{deg}(f_{*} \mathscr{F})-\dfrac{\alpha}{h}(1-g)$, so $\textrm{deg}( \mathscr{F})\leq \alpha(1-g+\tau_{\textrm{max}}( \mathscr{E})+n)-\dfrac{\alpha}{h}(1-g)$ and, therefore, $\mu( \mathscr{F})\leq 1-g+\tau_{\textrm{max}}( \mathscr{E})+n-\dfrac{1}{h}(1-g)$. Finally, from Lemma \ref{uniform1}, there is a constant $S_{+}$, depending only on $g$ and $h$, such that $\tau_{\textrm{max}}( \mathscr{E})\leq S_{+}$. Therefore, $\mu( \mathscr{F})\leq 1-g+S_{+}+n-\dfrac{1}{h}(1-g)=:C$.
\end{proof}


\begin{lemma}\label{bound-tor}
Let $\underline{P}$ be a finite set of polynomials of degree one with integral coefficients. There is a natural number $B_{\tau}$ depending only on $\underline{P}$ and $c$ such that for every Cohen-Macaulay projective and connected curve with a very ample line bundle $\mathscr{O}_{X}(1)$ of degree $h$ and every coherent sheaf, $\mathscr{F}$, of pure dimension one  with Hilbert polynomial in $\underline{P}$ such that $\chi(T(\mathscr{F}))\leq c$ we have $\textrm{dim}(\mathscr{F}_{x}/\mathfrak{m}_{x}\mathscr{F}_{x})\leq B_{\tau}$.
\end{lemma}
\begin{proof}
Define $B_{\tau}$ as the maximum of the leading coefficients of the polynomials in $\underline{P}$.
Consider a finite surjective morphism $f:X\rightarrow \mathbf{P}^{1}$ such that $f^{*}\mathscr{O}_{\mathbf{P}^{1}}(1)\simeq \mathscr{O}_{X}(1)$. Since $f$ must be flat, $f_{*}\mathscr{F}$ is locally free of rank bounded from above by $B_{\tau}$. Since $\textrm{dim}(\mathscr{F}_{x}/\mathfrak{m}_{x}\mathscr{F}_{x})\leq \textrm{rk}(f_{*}\mathscr{F})$, we are done.
\end{proof}


 \begin{theorem}\label{bound3}
Let $g,h,C\in\mathbb{N}$ with $g\geq 2$, and $\underline{P}$ a finite set of polynomials of degree one with integral coefficients. There is a natural number $N_{0}$ depending only on $\underline{P}, C,g,h$ such that for every Cohen-Macaulay projective and connected curve of genus $g$ with a very ample line bundle $\mathscr{O}_{X}(1)$ of degree $h$ and every coherent sheave of pure dimension one over $X$, $\mathscr{F}$, with Hilbert polynomial in $\underline{P}$ that satisfy $\mu_{\textrm{max}}(\mathscr{F})\leq C$, we have $h^{1}(X,\mathscr{F}(k))=0$ for all $k\geq N_{0}$ and $\mathscr{F}(k)$ is generated by its global sections. 
 \end{theorem}
\begin{proof}

Let $X$, $\mathscr{O}_{X}(1)$ and $\mathscr{F}$ be as in the statement. By Serre duality theorem, we have that
$
h^{1}(X,\mathscr{F}(k))=\textrm{dim}(\textrm{Hom}_{\mathscr{O}_{X}}(\mathscr{F}(k),\mathscr{\omega}_{X})), \  \forall k\in\mathbb{Z}.
$
Suppose that $h^{1}(X,\mathscr{F}(k))\neq0$ and $k\geq 0$. Then, there is a nonzero morphism 
$f':\mathscr{F}\longrightarrow \mathscr{\omega}_{X}(-k).$
Define $\mathscr{K}:=\textrm{Ker}(f')$ and $\mathscr{N}:=\textrm{Im}(f')$ and consider the exact sequence
$$
\xymatrix{
0\ar[r] & \mathscr{K}\ar[r] & \mathscr{F}\ar[r] & \mathscr{N} \ar[r] & 0.
}
$$
Then, we have
\begin{align*}
\textrm{deg}(\mathscr{N})&=\textrm{deg}(\mathscr{F})-\textrm{deg}(\mathscr{K})=\textrm{deg}(\mathscr{F})-\alpha_{\mathscr{K}}\mu(\mathscr{K})\geq  \textrm{deg}(\mathscr{F})-\alpha_{\mathscr{K}}C= &\\
&= P_{ \mathscr{F}}(0)-\dfrac{P_{ \mathscr{F}}'(n)}{h}(1-g)-\alpha_{\mathscr{K}}C\geq\\
&\geq B_{0}:=\textrm{min}\{P_{ \mathscr{F}}(0)-\dfrac{P_{ \mathscr{F}}'(n)}{h}(1-g)-i \cdot C|  i\in[1,\alpha_{\textrm{max}}]\textrm{, }P_{ \mathscr{F}}(n) \in \underline{P}\}
\end{align*}
$\alpha_{\textrm{max}}$ being the maximum among the degree one coefficients of the polynomials in $\underline{P}$. Note that  $B_{0}$ is a constant which depends only on $\underline{P}$, $h$ and the genus $g$.
The injective morphism $\mathscr{N}\hookrightarrow \mathscr{\omega}_{X}(-k)$ induces an injective morphism $\mathscr{N}(k)\hookrightarrow \mathscr{\omega}_{X}$. Then, we have
$\textrm{deg}(\mathscr{N}(k))=\alpha_{\mathscr{N}}\cdot k+\textrm{deg}(\mathscr{N})\geq k+B_{0}$, 
and therefore,
\begin{equation}\label{ineq}
\begin{split}
2g-2&=\textrm{deg}(\mathscr{\omega}_{X})=\textrm{deg}(\mathscr{N}(k))+\textrm{deg}(\mathscr{\omega}_{X}/\mathscr{N}(k))\geq\\
&\geq k+B_{0}+\textrm{deg}(\mathscr{\omega}_{X}/\mathscr{N}(k)).
\end{split}
\end{equation}
Let us find a bound of $\textrm{deg}(\mathscr{\omega}_{X}/\mathscr{N}(k))$. Denote $\mathscr{J}:=\mathscr{\omega}_{X}/\mathscr{N}(k)$. Fix a finite surjective morphism $f:X\rightarrow \mathbf{P}^{1}$ such that $f^{*}\mathscr{O}_{\mathbf{P}^{1}}(1)\simeq \mathscr{O}_{X}(1)$.
Since $f$ is finite, we have a surjection
$
f_{*}(\mathscr{\omega}_{X})\twoheadrightarrow f_{*}\mathscr{J}. 
$
Denote by $\mathscr{T}$ the torsion subsheaf of $f_{*}\mathscr{J}$ and by $\mathscr{U}=f_{*}\mathscr{J}/\mathscr{T}$ the locally free subsheaf. If $\mathscr{U}=0$, $\textrm{deg}(f_{*}\mathscr{I})\geq 0$. On the other hand, if $\mathscr{U}\neq 0$, we have
$f_{*}\mathscr{J}=\mathscr{U}\oplus\mathscr{T}$, so $\textrm{deg}(f_{*}\mathscr{J})\geq \textrm{deg}(\mathscr{U})$, and it is enough to give a bound for $\mathscr{U}$.
Observe that we have a surjective morphism $f_{*}(\mathscr{\omega}_{X})\twoheadrightarrow \mathscr{U}$, and that 
\begin{align*}
f_{*}(\mathscr{\omega}_{X})=\bigoplus_{i=1}^{h}\mathscr{O}(a_{i}), \
\mathscr{U}=\bigoplus_{i=1}^{T}\mathscr{O}(b_{j}).
\end{align*}
From the above surjection, we deduce that for every $j=1,\hdots T$ there exists an index $i=1,\hdots h$ such that $b_{j}\geq a_{i}$. Therefore,
$$
\textrm{deg}(\mathscr{U})=\sum_{i=1}^{T}b_{i}\geq \sum_{l=1}^{T}a_{i_{k}}\geq T\cdot A,
$$
where $A:=\textrm{min}\{a_{i}\}$. Now, by Proposition \ref{uniform1}, there exists an integer $S=S(g,h,m)$ depending only on $g,h,m:=h^{0}(X,\mathscr{\omega}_{X})=g$ such that $A\geq S$. Therefore, $\textrm{deg}(\mathscr{U})\geq T\cdot S\geq T_{0}\cdot S$ where $T_{0}:=\textrm{min}\{0,h\cdot S\}$. Set $C_0=\textrm{min}\{T_o\cdot S,0\}$. Then, from Equation (\ref{ineq}), we get
\begin{equation}\label{ineqmu}
2g-2\geq k+B_{0}+C_{0}.
\end{equation}
Let $N_{0}'$ be the smaller integer such that $N_{0}'+B_{0}+C_{0}>2g-2$. Then $h^{1}(X,\mathscr{F}(k))=0$ for all $k\geq N_{0}'$ and $N_{0}'$ only depends on $P, C,g,h$.

For the last part, let $x\in X$ be a closed point and $\mathscr{I}_{x}$ its ideal sheaf. Define $\mathscr{G}:=\textrm{Im}(\mathscr{F}\otimes_{\mathscr{O}_{X}}\mathscr{I}_{x}\rightarrow \mathscr{F})$. We have $\mu_{max}(\mathscr{G})\leq C$ and 
$$
P_{\mathscr{G}}(n)=P_{\mathscr{F}}(n)-d(x), \textrm{ where }d(x):=\textrm{dim}\mathscr{F}_{x}/\mathfrak{m}_{x}\mathscr{F}_{x}.
$$
By Lemma \ref{bound-tor}, the function $d(x)$ is bounded from above by a constant $B_{\tau}$ that depends only on $\underline{P}$. We can argue now as above and we arrive at the equation
$$
2g-2\geq k+B_{0}-B_{\tau}+C_{0}.
$$
Let $N_{0}(\geq N_{0}')$ be the smallest integer such that $N_{0}+B_{0}-B_{\tau}+C_{0}>2g-2$. Then $h^{1}(X,\mathscr{F}(-x)(k))=h^{1}(X,\mathscr{F}(k))=0$ for all $k\geq N_{0}$, so $\mathscr{F}(k)$ is generated by its global sections and $h^{1}(X,\mathscr{F}(k))=0$ for every $k\geq N_{0}$ and $N_{0}$ only depends on $\underline{P}, C,g,h$.
\end{proof}

 \begin{corollary}\label{bound4}
Let $g,h,C,C'\in\mathbb{N}$ with $g\geq 2$ and $\underline{P}$ a finite set of polynomials of degree one with integral coefficients. There is a natural number $N_{1}$ depending only on $g,h,C,C',\underline{P}$ such that for every Cohen-Macaulay projective and connected curve of genus $g$ with a very ample line bundle $\mathscr{O}_{X}(1)$ of degree $h$, every family $E$ of equivalence classes of coherent sheaves of pure dimension one, $\mathscr{F}$, with Hilbert polynomial in $\underline{P}$ that satisfy $\mu_{\textrm{max}}(\mathscr{F})\leq C$ and every family $E'$ of equivalence classes of subsheaves $ \mathscr{F}'\subset \mathscr{F}$, with $ \mathscr{F}\in E$, such that $|\emph{deg}( \mathscr{F}')|\leq C'$, we have $h^{1}(X,\mathscr{F'}(k))=0$ for all $k\geq N_{1}$ and $\mathscr{F'}(k)$ is generated by its global sections. Therefore, $\emph{reg}(\mathscr{F'})\leq N_{1}$ and the family is bounded.
 \end{corollary}
\begin{proof}
Let $ \mathscr{F}\in E$ and $ \mathscr{F}'\subset  \mathscr{F}$ a coherent sheaf in $E'$. We have $\mu_{\textrm{max}}( \mathscr{F}')\leq C$. Since $|\textrm{deg}( \mathscr{F}')|\leq C'$, there are only finitely many possible polynomials in the set of Hilbert polynomials of the members of $E'$. Then,  applying Theorem \ref{bound3}, we conclude.
\end{proof}

 \begin{corollary}\label{bound5}
Let $g,h,C,a\in\mathbb{N}$ with $g\geq 2$ and $\underline{P}$ a finite set of polynomials of degree one with integral coefficients. There is a natural number $N_{1}'$ depending only on $g,h,C,a,\underline{P}$ such that for every Cohen-Macaulay projective and connected curve of genus $g$ with a very ample line bundle $\mathscr{O}_{X}(1)$ of degree $h$ and every family $E$ of equivalence classes of coherent sheaves of pure dimension one, $\mathscr{F}$, with Hilbert polynomial in $\underline{P}$ that satisfy $\mu_{\textrm{max}}(\mathscr{F})\leq C$
we have $h^{1}(X,(\mathscr{F}_1\otimes\cdots\otimes\mathscr{F}_a)(k))=0$ for all $k\geq N_{1}'$ and $(\mathscr{F}_1\otimes\cdots\otimes\mathscr{F}_a)(k)$ is generated by its global sections. Therefore, $\emph{reg}(\mathscr{F}_1\otimes\cdots\otimes\mathscr{F}_a)\leq N_{1}'$ and the family $E^{\otimes a}$ is bounded.
 \end{corollary}
\begin{proof}
By Theorem \ref{bound3},
it is enough to show that if $\mathscr{F}, \mathscr{G}$ are coherent sheaves of $\mathscr{O}_{X}$-modules that are $k_1$- and $k_2$-regular respectively, then $ \mathscr{F}\otimes \mathscr{G}$ is $(k_1+k_2)$-regular. Let $j:X\hookrightarrow \mathbb{P}^{n}$ be a closed immersion such that $j^{*}\mathscr{O}_{\mathbb{P}^{n}}(1)=\mathscr{O}_{X}(1)$. Since $j:X\hookrightarrow \mathbb{P}^{n}$ is a finite morphism, we know that $H^{i}(X, \mathscr{F}(n))=H^{i}(\mathbb{P}^{n},j_{*}( \mathscr{F})(n))$ and $H^{i}(X, \mathscr{G}(n))=H^{i}(\mathbb{P}^{n},j_{*}( \mathscr{G})(n))$. This implies that $j_{*} \mathscr{F}$ and $j_{*} \mathscr{G}$ are $k_1$- and $k_2$-regular respectively, and obviously $(j_{*} \mathscr{F})_{y}=(j_{*} \mathscr{G})_{y}=0$ for every $y\in\mathbb{P}^{n}\setminus X$. Therefore, by \cite[Proposition 1.5]{sidman}, $j_{*} \mathscr{F}\otimes j_{*} \mathscr{G}$ is $(k_1+k_2)$-regular. That is, 
$$H^{1}(\mathbb{P}^{n},(j_{*} \mathscr{F}\otimes j_{*} \mathscr{G})(k_1+k_2-1))=0$$ 
for every $i\geq 0$. Since $j$ is a closed immersion, we deduce that $j_{*} \mathscr{F}\otimes j_{*} \mathscr{G}=j_{*}( \mathscr{F}\otimes  \mathscr{G})$ and, therefore,
\begin{align*}
\begin{split}
H^{1}(X,( \mathscr{F}\otimes \mathscr{G})(k_1+k_2-1))&=H^{1}(\mathbb{P}^{n},j_{*}(( \mathscr{F}\otimes \mathscr{G})(k_1+k_2-1)))=\\
&=H^{1}(\mathbb{P}^{n},j_{*}( \mathscr{F}\otimes \mathscr{G}))(k_1+k_2-1))=\\
&=H^{1}(\mathbb{P}^{n},(j_{*} \mathscr{F}\otimes j_{*} \mathscr{G})(k_1+k_2-1))=0
\end{split}
\end{align*}
\end{proof}
Let $F$ be a finite dimensional vector space, $X$ a Cohen-Macaulay projective and connected curve of genus $g$,  $\mathscr{O}_{X}(1)$ a very ample line bundle of degree $h$ and let $k\in\mathbb{Z}$ be an integer. Given a quotient $q:F\otimes\mathscr{O}_{X}(k)\rightarrow \mathscr{F}$ and a vector subspace $F'\subset F$, we will denote by $\mathscr{F}_{F'}$ the subsheaf $q(F'\otimes\mathscr{O}_{X}(k))$.

\begin{corollary}\label{gitsemistability}
Let $F$ be a finite dimensional vector space, $k\in\mathbb{Z}$, $l_0\in\mathbb{N}$, $P(n)\in\mathbb{Z}[n]$ and  $X$ a Cohen-Macaulay projective and connected curve of genus $g\geq 2$ with a very ample line bundle $\mathscr{O}_{X}(1)$ with $h^{1}(X,\mathscr{O}_{X}(l_0))=0$. Then, there exists a natural number $L\in\mathbb{Z}$ depending only on $k,l_0,g,P(n)$ such that
for every quotient sheaf $q:F\otimes\mathscr{O}_{X}(k)\rightarrow \mathscr{F}$ with Hilbert polynomial $P(n)$ and for every vector subspace $F'\subset F$, $h^{1}(X,\mathscr{F}_{F'}(l))=0$ and $H^{0}(q(l))(F'\otimes W)=H^{0}(X, \mathscr{F}_{F'}(l))$ for every $l>L$, where $W=H^{0}(X,\mathscr{O}_{X}(l+k))$.
 \end{corollary}
\begin{proof}
Let $X$ be a Cohen-Macaulay projective and connected curve of genus $g\geq 2$ together with a very ample line bundle $\mathscr{O}_{X}(1)$ with the conditions of the statement, let $q:F\otimes\mathscr{O}_{X}(k)\rightarrow \mathscr{F}$ be a quotient sheaf with Hilbert polynomial $P(n)$ and let $F'\subset F$ be a vector subspace.
Consider the exact sequence 
$$
\xymatrix{
0\ar[r] & \mathscr{K}_{F'}(l)\ar[r]  & F'\otimes\mathscr{O}_{X}(l+k) \ar[r]& \mathscr{F}_{F'}(l) \ar[r] & 0
}
$$
Taking global sections we get a surjection
$$
F'\otimes H^{1}(X,\mathscr{O}_{X}(l+k))\longrightarrow   H^{1}(X,\mathscr{F}_{F'}(l))\longrightarrow 0
$$
Clearly $h^{1}(X,\mathscr{O}_{X}(p))=0$ for every $p\geq l_0$ since $h^{1}(X,\mathscr{O}_{X}(l_0))=0$. Therefore, $l>-k+l_0$ implies $h^{1}(X,\mathscr{F}_{F'}(l))=0$. On the other hand, it follows from Lemma \ref{bound2-3} that every sheaf $\mathscr{K}$ of the form $\textrm{Ker}(q)$ for some $q:F\otimes\mathscr{O}_{X}(k)\rightarrow \mathscr{F}$ as in the statement satisfies that $\mu_{max}(\mathscr{K})$ is bounded from above by a constant depending only on the numerical input data and there is a finite set $\underline{P}$ of polynomials with integral coefficients, depending only on the numerical input data as well, to which $P_{\mathscr{H}}(n)$ belongs. Now, by Theorem \ref{bound3}, we deduce that there is a constant $l_{1}\in\mathbb{N}$, depending only on the numerical input data such that for every $l\geq l_{1}$,  $h^{1}(X,\mathscr{K}_{F'}(l))=0$ and $\mathscr{K}_{F'}(l)$ is generated by its global sections. Thus, for every $l\geq L:=\textrm{max}\{-k+l_0,l_1\}$ we have $H^{0}(q(l))(F'\otimes W)=H^{0}(X, \mathscr{F}_{F'}(l))$.
\end{proof}

\subsection{Uniform boundedness of $\delta$-semistable swamps}


Let $X$ be a Cohen-Macaulay projective and connected curve of genus $g$. We define $\delta$-semistability for swamps by simply substituting the ranks by the multiplicities of the coherent sheaves \cite{AMC}. A direct application of Theorem \ref{bound3} will allow us to prove Theorem \ref{bound10} which, in turn, will imply that the families of $\delta$-semistable swamps of fixed type are uniformly bounded along with the moduli space of stable curves of genus $g$.

 \begin{definition}
Let $a,b\in\mathbb{N}$ and $\mathscr{D}$ an invertible sheaf. A swamp over $X$ of type $(a,b, \mathscr{D})$ is a pair $( \mathscr{F},\phi)$ where $ \mathscr{F}$ is a coherent $\mathscr{O}_{X}$-module and $\phi$ is a non-zero  morphism of $\Ox$-modules, $\phi\colon ( \mathscr{F}^{\otimes a})^{\oplus b}\rightarrow \mathscr{D}$.
 \end{definition}

 \begin{definition}
Let $ \mathscr{F}$ be a coherent $\mathscr{O}_{X}$-module on $X$. A weighted filtration, $( \mathscr{F}_{\bullet},\underline{m})$, of $ \mathscr{F}$ is a filtration
$ \mathscr{F}_{\bullet}\equiv (0)\subset \mathscr{F}_{1}\subset \mathscr{F}_{2}\subset\hdots\subset  \mathscr{F}_{t}\subset \mathscr{F}_{t+1}= \mathscr{F}$,
equipped with positive numbers $m_{1}\hdots,m_{t}\in\mathbb{Q}_{>0}$.
 \end{definition}
We adopt the following convention: the one-step filtration is always equipped with $m=1$. A filtration is called saturated if the quotients $ \mathscr{F}/ \mathscr{F}_{i}$ are coherent sheaves of pure dimension one.
 
Let $\phi\colon ( \mathscr{F}^{\otimes a})^{\oplus b}\rightarrow \mathscr{D}$ be a swamp on $X$ and let $( \mathscr{F}_{\bullet},\underline{m})$ be a weighted filtration. For each $ \mathscr{F}_{i}$ denote by $\alpha_{i}$ its multiplicity  and just by $\alpha$ the multiplicity of $ \mathscr{F}$. Define the vector 
\begin{equation*}
  \Gamma=\sum_{1}^{t}m_{i}  \Gamma^{(\alpha_{i})},
\end{equation*}
where $  \Gamma^{(l)}=(\overset{l}{\overbrace{l-\alpha,\hdots,l-\alpha}},\overset{\alpha-l}{\overbrace{l,\hdots,l}})$. Let us denote by $J$ the set
\begin{equation*}
J=\{\textrm{ multi-indices }I=(i_{1},\hdots,i_{a})|I_{j}\in\{1,\hdots,t+1\}\}.
\end{equation*}
Define
\begin{equation}
\mu( \mathscr{F}_{\bullet},\underline{m},\phi)= \textrm{min}_{I\in J}\{  \Gamma_{\alpha_{i_{1}}}+\hdots +  \Gamma_{\alpha_{i_{a}}}|\phi|_{( \mathscr{F}_{i_{1}}\otimes \hdots\otimes  \mathscr{F}_{i_{a}})^{\oplus b}}\neq 0\}.
\end{equation}

 \begin{definition}\label{semitensor}
Let $\delta\in\mathbb{Q}_{>0}$ be a positive rational number. A swamp $( \mathscr{F},\phi)$ of type $(a,b,\mathscr{D})$ is $\delta$-(semi)stable if for each weighted filtration $( \mathscr{F}_{\bullet},\underline{m})$ the following holds
\begin{equation}\label{semiestabilidad}
\sum_{1}^{t}m_{i}(\alpha P_{ \mathscr{F}_{i}}-\alpha_{i}P)+ \delta\mu( \mathscr{F}_{\bullet},\underline{m},\phi) (\leq)0.
\end{equation}
 \end{definition}

\begin{remark} \label{finiteset}
There is a positive integer $A$, depending only on the numerical input data $P,a$, such that it is enough to check the $\delta$-semistability condition (\ref{semiestabilidad}) for weighted filtrations with $m_{i}<A$.
Note that a swamp is $\delta$-(semi)stable if and only if Equation (\ref{semiestabilidad}) holds for every integral weighted filtration, i.e., filtrations with integral weights. Now, the claim follows from \cite[Lemma 1.4]{sols} changing ranks by multiplicities. Observe that the upper bound $A$ does not depend either on $b$ or on the sheaf $ \mathscr{D}$.
 \end{remark}

 \begin{theorem}\label{bound10}
Let $a,g,h\in\mathbb{N}$ with $g\geq 2$ and let $P(n)\in\mathbb{Z}[n]$ be a polynomial of degree one and $\delta\in\mathbb{Q}_{>0}$. There exists a natural number $N_{2}\in\mathbb{N}$ depending only on $P(n),a,\delta,g,h$ such that for every Cohen-Macaulay projective and connected curve of genus $g$, $X$, with very ample line bundle $\mathscr{O}_{X}(1)$ of degree $h$ and for every coherent sheaf of pure dimension one $ \mathscr{F}$ with Hilbert polynomial $P(n)$ appearing in a $\delta$-(semi)stable swamp of type $(a,-,-)$, $\mathscr{F}(k)$ is generated by its global sections and $h^{1}(X,\mathscr{F}(k))=0$ for every $k\geq N_{2}$. In particular, the set $E$ of equivalence classes of coherent sheaves of pure dimension one with Hilbert polynomial $P(n)$ appearing in $\delta$-semistable swamps of type $(a,-,-)$ is bounded.
 \end{theorem}
 \begin{remark}
The notation $(a,-,-)$ means that we allow swamps $\phi:( \mathscr{F}^{\otimes a})^{b}\rightarrow  \mathscr{D}$, whatever $b\in\mathbb{N}$ and $ \mathscr{D}\in\textrm{Pic}(X)$ are.
 \end{remark}
\begin{proof}
Let $X$ be a Cohen-Macaulay projective and connected curve of genus $g$ $b\in\mathbb{N}$ a natural number, $ \mathscr{D}$ an invertible sheaf and $( \mathscr{F},\phi)$ a $\delta$-semistable swamp of type $(a,b, \mathscr{D})$ with Hilbert polynomial $P(n)$. Let $ \mathscr{F}_{1}\subset  \mathscr{F}$ be a subsheaf and consider the one-step flag $0\subset \mathscr{F}_{1}\subseteq \mathscr{F}_{2}= \mathscr{F}$. The computation of the semistability condition given in  Equation (\ref{semiestabilidad}) leads to
$$\mu( \mathscr{F}_{1})\leq C:=\mu( \mathscr{F})+\dfrac{a(\alpha-1)}{\alpha }\delta$$
$C$ being a constant depending only on $P,h,a,\delta,g$. Now, by Theorem \ref{bound3}, there exists a natural number $N_{2}\in\mathbb{N}$ depending only on $P(n), C, g, h$ (thus, on $P(n),a,$ $\delta,g,h$) such that $h^{1}(X, \mathscr{F}(k))=0$ and $ \mathscr{F}(k)$ is generated by its global sections for every $k\geq N_{2}$. Finally, by  (\cite[Lemma 1.7.6]{Huyb}, the set $E$ of equivalence classes of coherent sheaves of pure dimension one with Hilbert polynomial $P(n)$ appearing in $\delta$-semistable swamps of type $(a,-,-)$ is bounded.
\end{proof}

\subsection{Comparision of $\delta$-semistability and GIT semistability}\label{comparision}
Our goal is to prove Theorem \ref{teoremon}.
The construction of the moduli space of swamps, $\mathcal{T}^{\delta\text{-(s)s}}_{X,P,a,b,\mathscr{D}}$, for a single stable curve $X$ of genus $g$ has been done in \cite{L} for any a projective scheme of pure dimension one. However, we will go over the key steps of such construction to exhibit how Theorem \ref{bound3} shows the uniformity along with $\overline{\textrm{M}}_{g}$ of the parameters $N,L$ involved in the construction. This, together with \cite[Theorem 5.4]{AMC}, eventually shows the main property of the projective scheme that will be constructed.

Let $a,b,g,h,l_0\in\mathbb{N}$ with $g\geq 2$ and let $P(n)\in\mathbb{Z}[n]$ be a polynomial of degree one and $\delta\in\mathbb{Q}_{>0}$
Let $X$ be a Cohen-Macaulay projective and connected curve of genus $g$ with a very ample line bundle of degree $h$ such that $h^{1}(X,\mathscr{O}_{X}(l_0))=0$ for a fixed natural number $l_o\in\mathbb{N}$. Let  $D\in\mathbb{N}$ be a natural number, $e\in\mathbb{Z}$ an integer and let $ \mathscr{D}$ be an invertible sheaf on $X$ of degree $e$ with $h^{0}(X, \mathscr{D})=D$.  Given $k\in\mathbb{N}$, let $\mathcal{H}\subset\textrm{Quot}^{P}_{\mathbb{C}^{P(k)}\otimes\mathscr{O}_{X}(-k)/X/\mathbb{C}}$ be the subscheme of the Quot scheme parametrizing quotients $F\otimes\Ox(-k)\rightarrow  \mathscr{F}$ of pure dimension one, with $F=\mathbb{C}^{P(k)}$. By Theorem \ref{bound3} and  Lemma \ref{uniform1}, there is a natural number $N'\in\mathbb{Z}$, depending only on the numerical input data, such that for every $k\geq N'$ and for each $\delta$-(semi)stable swamp, $( \mathscr{F},\phi)$, of type $(a,b, \mathscr{D})$, we have that $ \mathscr{F}(k)$ and $\mathscr{D}(k)$ are generated by its global sections, and  $h^{1}(X, \mathscr{F}(k))=h^{1}(X, \mathscr{D}(k))=0$.

Fix such a natural number $k>N'$ and set $F=\mathbb{C}^{P(k)}$. Let $L$ be as in Lemma \ref{gitsemistability}. For $l\geq L$, there is a projective embedding (Grothendieck embedding composed with the Pl\"ucker embedding),
\begin{equation*}
\mathcal{H}\hookrightarrow \mathbf{P}(\bigwedge^{P(l)}(F\otimes H^{0}(X,\Ox(l-k)))).
\end{equation*}
Consider the vector space  $Z_1=((F^{\otimes a})^{\oplus b})^{\vee}\otimes H^{0}(X, \mathscr{D}(ak))$.
The functor of points of the projective space $\mathbf{P}(Z_1)$ is given by
\begin{equation}
\mathbf{P}(Z_1)^{\bullet}(T)=\left\{
\begin{array}{l}
\textrm{ equivalence classes of invertible quotients }\\
((F^{\otimes a})^{\oplus b}\otimes H^{0}(X, \mathscr{D}(ak))^{\vee})\otimes\mathcal{O}_{T}\rightarrow \mathcal{L} \textrm{ over } T \\
\end{array}\right\}.
\end{equation}

Consider now  the functor (for $k$ fixed as in the introduction)
\begin{equation}\label{planta}
^{\textrm{rig}}\textbf{Swamps}^{k}_{P, \mathscr{D},a,b}(T)=\left\{
\begin{array}{l}
\textrm{isomorphism classes of tuples }( \mathscr{F}_{T},\phi_{T},\N,g_{T})\\
\textrm{where }( \mathscr{F}_{T},\phi_{T},\N) \textrm{ is a swamp with }\\
\textrm{Hilbert polynomial } P \textrm{ and }g_{T}\textrm{ is a morphism}\\
g_{T}\colon F\otimes\mathscr{O}_{T}\rightarrow \pi_{T*} \mathscr{F}_{T}(k)\textrm{ such that} \\
\textrm{the induced morphism }V\otimes\mathscr{O}_{X\times T}\rightarrow \mathscr{F}_{T}(k)\\
\textrm{is surjective }
\end{array}\right\},
\end{equation}
where two tuples, $( \mathscr{F}_{T},\phi_{T},\N,g_{T})$ and $( \mathscr{F}'_{T},\phi'_{T},\N',g'_{T})$, are isomorphic if there exists an isomorphism $(f,\alpha)$  between $( \mathscr{F}_{T},\phi_{T},\N)$ and $( \mathscr{F}'_{T},\phi'_{T},\N')$ such that $\pi_{T*}(f(k))\circ g_{T}=g'_{T}$.
With the same argument given in the case of irreducible curves (see \cite{bosle,sols}) it follows that the functor $^{\emph{rig}}\textbf{Swamps}^{k}_{P,\mathscr{D},a,b}$ is represented by a closed subscheme $W_{P,\mathscr{D},a,b}^{k,l}(X)$ of $\mathcal{H}\times\mathbf{P}(Z_1)$.
Let $Z_{P,\mathscr{D},a,b}^{k,l}(X)\subset W_{P,\mathscr{D},a,b}^{k,l}(X)$ be the closure of the locus representing $\delta$-semistable swamps. Consider the projections
\begin{align*}
p_{\mathcal{H}}&\colon Z_{P,\mathscr{D},a,b}^{k,l}(X)\rightarrow\mathcal{H}\\
p_{\mathcal{P}}&\colon Z_{P,\mathscr{D},a,b}^{k,l}(X)\rightarrow\mathbf{P}(Z_1)
\end{align*}
and define a polarization on $Z_{P,\mathscr{D},a,b}^{k,l}(X)$ by
\begin{equation}
\mathscr{O}_{Z_{k,\mathscr{D}}(X)}(n_{1},n_{2}):=p_{\mathcal{H}}^{*} \mathcal{O}_{\mathcal{H}}(n_{1})\otimes p_{\mathbf{P}(Z_1)}^{*}\mathscr{O}_{\mathbf{P}(Z_1)}(n_{2}),
\end{equation}
$n_{1}$ and $n_{2}$ being positive integers such that
\begin{equation}\label{polarization}
\dfrac{n_{1}}{n_{2}}=\dfrac{P(l)-\textrm{dim}(F)}{\textrm{dim}(F)-s\delta} \delta.
\end{equation}
The natural action of $\textrm{SL}(F)$ on $\mathcal{H}\times\mathbf{P}(Z_1)$ preserves $Z_{P,\mathscr{D},a,b}^{k,l}(X)$ and the linearizations on $\mathscr{O}_{\mathcal{H}}(1)$ and $\mathscr{O}_{\mathbf{P}(Z_1)}(1)$ induce a linearization on $\mathscr{O}_{Z_{P,\mathscr{D},a,b}^{k,l}(X)}(n_{1},n_{2})$. Then, the GIT quotient $Z_{P,\mathscr{D},a,b}^{k,l}(X)\slas \textrm{SL}(F)$ exists and is projective, and if $k,l$ are large enough it is a coarse moduli space for $\delta$-semistable swamps of the corresponding type \cite{L}.

Our aim is to show the following result, which shows that $k,l$ ca be chosen so that they work for every Cohen-Macaulay projective and connected curve.

 \begin{theorem}\label{teoremon}
Let $g,h,D,a,b\in\mathbb{N}$, with $g\geq 2$ and $P(n)\in\mathbb{Z}[n]$ a polynomial of degree one . There are natural numbers $N \ (\geq N'),L\in\mathbb{N}$ depending only on the numerical input data such that for every $k\geq N$ and every $l\geq L$ the following holds:
for every Cohen-Macaulay projective and connected curve of genus $g$ with a very ample line bundle $\mathscr{O}_{X}(1)$ of degree $h$, the point $(q,\Phi)\in Z_{P,\mathscr{D},a,b}^{k,l}(X)$ is $\emph{GIT}$-(semi)stable with respect to $\mathscr{O}_{Z_{P,\mathscr{D},a,b}^{k,l}(X)}(n_{1},n_{2})$ if and only if the corresponding swamp $( \mathscr{F},\phi)$ is $\delta$-(semi)stable and the linear map $f_{q}\colon F\rightarrow H^{0}(X, \mathscr{F}(k))$ is an isomorphism, where $F:=\mathbb{C}^{P(k)}$.
 \end{theorem}

We will go over the key results that allows to prove Theorem \ref{teoremon} skipping some details to avoid repetition. We refer to the reader to \cite{sols} and \cite{AMCdis} for a detailed exposition of the next results.

\subsubsection{Step 1: $\delta$-semistability and sectional semistability}

Let $X$ be a Cohen-Macaulay projective and connected curve of genus $g$, $ \mathscr{F}$ a coherent $\mathscr{O}_{X}$-module, and suppose we have a filtration, $ \mathscr{F}_{\bullet}$, of $ \mathscr{F}$. We will denote by $\alpha^{i}$ the multiplicity of $ \mathscr{F}/ \mathscr{F}_{i}$ and by $\alpha_{i}$ the multiplicity of $ \mathscr{F}_{i}$ (thus, $\alpha( \mathscr{F})=\alpha_{i}+\alpha^{i}$). Let now $P(n)\in\mathbb{Z}[n]$ be a polynomial of degree one, $\alpha, \ d$ rational numbers such that $P(n)=\alpha n+\dfrac{\alpha}{h}(1-g)+d$, and $k$ a natural number. Then, we define:
\begin{enumerate}
\item $S^{s}$ is the set of $\delta$-semistable swamps $( \mathscr{F},\phi)$ with $ \mathscr{F}$ a coherent sheaf of pure dimension one with Hilbert polynomial $P$.

\item $S'_{k}$ is the set of swamps $( \mathscr{F},\phi)$ with $ \mathscr{F}$ a coherent sheaf of pure dimension one with Hilbert polynomial $P$, and such that 
\begin{equation*}
\sum_{i=1}^{t}m_{i}(\alpha h^{0}(X, \mathscr{F}_{i}(k))-\alpha_{i}P(k)))+\delta\mu( \mathscr{F}_{\bullet},\underline{m},\phi)\leq 0
\end{equation*}
for every weighted filtration $( \mathscr{F}_{\bullet},\underline{m})$.

\item $S''_{k}$ is the set of swamps $( \mathscr{F},\phi)$ with $ \mathscr{F}$ a coherent sheaf of pure dimension one with Hilbert polynomial $P$, and such that 
\begin{equation*}
\sum_{i=1}^{t}m_{i}(\alpha^{i}P(k)-\alpha h^{0}(X, \mathscr{F}^{i}(k)))+\delta\mu( \mathscr{F}_{\bullet},\underline{m},\phi)\leq 0,
\end{equation*}
for every weighted filtration $( \mathscr{F}_{\bullet},\underline{m})$.

\item $S_{N}=(\cup_{k\geq N}S''_{k})\cup S^{s}$, $N\in\mathbb{N}$.

\end{enumerate}

  \begin{lemma}\label{lemma 1,2}
Let $a,g,h\in\mathbb{N}$ with $g\geq 2$ and let $P(n)\in\mathbb{Z}[n]$ be a polynomial of degree one and $\delta\in\mathbb{Q}_{>0}$. There exists  natural numbers $N_{3},C_{0}\in\mathbb{N}$ depending only on $P(n),a,\delta,g,h$ such that for every Cohen-Macaulay projective and connected curve of genus $g$, $X$, with very ample line bundle $\mathscr{O}_{X}(1)$ of degree $h$, if $( \mathscr{F},\phi)\in S_{N}$ is a swamp of type $(a,-,-)$ then, for all saturated weighted filtrations $( \mathscr{F}_{\bullet},\underline{m})$ and for all $C\geq C_0$, the following holds for all $i$:
\begin{equation*}
\emph{deg}( \mathscr{F}_{i})-\alpha_{i}\mu_{s}\leq C,\textrm{ where } \mu_{s}:=\dfrac{d-a\delta}{\alpha}\textrm{ and }d:=\emph{deg}( \mathscr{F})
\end{equation*}
and either $1) \ -C\leq \emph{deg}( \mathscr{F}_{i})-\alpha_{i}\mu_{s}$, or\\
2.a) $h^{0}(X, \mathscr{F}_{i}(k))<\alpha_{i}(P(k)-a\delta)$, if $( \mathscr{F},\phi)\in S^{s}$ and $k\geq N_{3}$\\
2.b) $\alpha^{i}(P-a\delta)\preceq\alpha(P_{ \mathscr{F}^{i}}-a\delta)$ if $( \mathscr{F},\phi)\in\cup_{k\geq N_{1}}S''_{k}$.
 \end{lemma}
\begin{remark}
The symbol refers $\preceq$ to the lexicographic order of polynomials.
\end{remark}
\begin{proof}
It follows as in \cite[Lemma 2.6]{sols}. A short calculation gives an explicit expression for $N_3$ and $C_0$:
\begin{equation}
\begin{split}
C_{0}=[\textrm{max}\{a\delta,\alpha^{2}+B-r(1-g)+d\}]+1,\ N_{3}=[\textrm{max}\{0,B'+\dfrac{C_0}{\alpha}-\mu_{s}\}]+1,
\end{split}
\end{equation}
where $[- ]$ denotes de integral part, $B$ is the constant given in \cite[Corollary 1.7]{simpson}, which depends only on $P$ and $h$,  and $B':=B+\dfrac{(1-g)}{h}$.
\end{proof}
Lemma \ref{lemma 1,2} gives rise to the following definition.
\begin{enumerate}
\item[5.]$S_{0}$ is the set of saturated subsheaves, $ \mathscr{F}'\subset  \mathscr{F}$, of coherent sheaves, $ \mathscr{F}$, appearing  in swamps $( \mathscr{F},\phi)\in S_{N}$, and satisfying $|\textrm{deg}( \mathscr{F}')-\alpha'\mu_{s}|\leq C$ for all $C\geq C_{0}$.

\end{enumerate}

  \begin{lemma}\label{lemaSNS0}
Let $a,g,h\in\mathbb{N}$ with $g\geq 2$, $P(n)\in\mathbb{Z}[n]$ a polynomial of degree one and $\delta\in\mathbb{Q}_{>0}$. There exists a natural number $N_{4}\in\mathbb{N}$ depending only on $P(n),a,\delta,g,h$ such that for every Cohen-Macaulay projective and connected curve of genus $g$, $X$, with very ample line bundle $\mathscr{O}_{X}(1)$ of degree $h$  the following holds: for every swamp $( \mathscr{F},\phi)\in S_{N}$ on $X$ of type $(a,-,-)$ and for every subsheaf $ \mathscr{F}'\subset  \mathscr{F}$ with $ \mathscr{F}'\in S_{0}$, both, $ \mathscr{F}(k)$ and $ \mathscr{F}'(k)$, are generated by global sections and $h^{1}(X, \mathscr{F}(k))=h^{1}(X, \mathscr{F}'(k))=0$ for every $k\geq N_{4}$.
 \end{lemma}
\begin{proof}
If a swamp $( \mathscr{F},\phi)$ belongs to $S_{N}$, it holds
$$
\mu_{\textrm{max}}( \mathscr{F})\leq \mu_{s}+C +\dfrac{1-g}{h}, \textrm{ where }C\geq C_{0}.
$$
This implies that, for every swamp $( \mathscr{F},\phi)\in S_{N}$ and every subsheaf $ \mathscr{F}'\subset \mathscr{F}$, the inequality $\mu_{\textrm{max}}( \mathscr{F}')\leq \mu_{s}+C +\dfrac{1-g}{h}$ holds. On the other hand, every swamp $( \mathscr{F},\phi)\in S_{N}$ has Hilbert polynomial $P(n)$ and the set of Hilbert polynomials of subsheaves $ \mathscr{F}'\subset \mathscr{F}$ is finite. Moreover, the coefficients of these Hilbert polynomials have lower and upper bounds that depend only on $P(n),a,\delta,g,h$. Therefore, we conclude by applying Theorem \ref{bound3}.
\end{proof}

Then, following \cite[Theorem 2.5]{sols} and applying Corollary \ref{bound5} and Lemma \ref{lemaSNS0} we have:
 \begin{theorem}\label{sectional}
Let $a,g,h\in\mathbb{N}$, with $g\geq 2$, $P(n)\in\mathbb{Z}[n]$ a polynomial of degree one and $\delta\in\mathbb{Q}_{>0}$. There exists a natural number $N_{5}\in\mathbb{N}$ depending only on $P(n),a,\delta,g,h$ such that for every Cohen-Macaulay projective and connected curve of genus $g$, $X$, with very ample line bundle $\mathscr{O}_{X}(1)$ of degree $h$  the following properties of swamps, $( \mathscr{F},\phi)$, of type $(a,-,-)$ with $ \mathscr{F}$ a coherent sheaf of pure dimension one and with Hilbert polynomial $P(n)$, are equivalent:

1) $( \mathscr{F},\phi)$ is $\delta$-(semi)stable.

2) $ \forall \ ( \mathscr{F}_{\bullet},\underline{m})$ we have $\sum_{1}^{t}m_{i}(\alpha h^{0}(X, \mathscr{F}_{i}(k))-\alpha_{i}P(k))+\delta\mu( \mathscr{F}_{\bullet},\underline{m},\phi)(\leq)0.$

3) $ \forall \ ( \mathscr{F}_{\bullet},\underline{m})$ we have $\sum_{1}^{t}m_{i}(\alpha^{i}P(k)-\alpha h^{0}(X, \mathscr{F}^{i}(k)))+\delta\mu( \mathscr{F}_{\bullet},\underline{m},\phi)(\leq)0.$\\
Furthermore, for any swamp satisfying these conditions, we have $h^{1}(X, \mathscr{F}(k))=0$.
 \end{theorem}

 \begin{corollary}\label{strictsemi}
Let $a,g,h\in\mathbb{N}$, with $g\geq 2$, $P(n)\in\mathbb{Z}[n]$ a polynomial of degree one and $\delta\in\mathbb{Q}_{>0}$. Let $X$ be Cohen-Macaulay projective and connected curve of genus $g$, $X$, with very ample line bundle $\mathscr{O}_{X}(1)$ of degree $h$, $( \mathscr{F},\phi)$ a $\delta$-semistable swamp of type $(a,-,-)$, $k\geq N_{5}$, and assume that there is a weighted filtration $( \mathscr{F}_{\bullet},\underline{m})$ such that
\begin{equation}\label{canita}
(\sum_{i=1}^{t}m_{i}(\alpha h^{0}(X, \mathscr{F}_{i}(k))-\alpha_{i} P(k)))+\delta\mu( \mathscr{F}_{\bullet},\underline{m},\phi)=0.
\end{equation}
Then $ \mathscr{F}_{i}\in S_{0}$ and $h^{0}(X, \mathscr{F}_{i}(k))=P_{ \mathscr{F}_{i}}(k)$ for all $i$.
 \end{corollary}
\begin{proof}
Let $k\geq N_5$. If $ \mathscr{F}_{i}\in S_{0}$, then $P_{ \mathscr{F}_{i}}(k)=h^{0}(X, \mathscr{F}_{i}(k))$ by Lemma \ref{lemaSNS0}. If $ \mathscr{F}_{i}$ do not belongs to $S_{0}$, then the second alternative of Lemma \ref{lemma 1,2} holds, so
\begin{equation}\label{papafrita}
\alpha h^{0}(X, \mathscr{F}_{i}(k))<\alpha_{i}(P(k)-a\delta).
\end{equation}
Let $T'\subset T=\{1,\hdots,t\}$ be the subset of those $i$ for which $ \mathscr{F}_{i}\in S_{0}$. Let $( \mathscr{F}'_{\bullet},\underline{m}')$ be the corresponding subfiltration. Then
\begin{equation}\label{canita2}.   
\begin{split}
&(\sum_{i=1}^{t}m_{i}(\alpha h^{0}(X, \mathscr{F}_{i}(k))-\alpha_{i} P(k)))+\delta \mu( \mathscr{F}_{\bullet},\underline{m},\phi)\leq \\
\leq&(\sum_{i=1}^{t}m_{i}(\alpha h^{0}(X, \mathscr{F}_{i}(k))-\alpha_{i} P(k)))+\delta  \mu( \mathscr{F}'_{\bullet},\underline{m}',\phi)+ \delta(\sum_{i\in T\setminus T'}m_{i}a\alpha_{i})=\\
=&(\sum_{i\in T'}m_{i}(\alpha h^{0}(X, \mathscr{F}_{i}(k))-\alpha_{i} P(k)))+\delta  \mu( \mathscr{F}'_{\bullet},\underline{m}',\phi)+ \\ +&(\sum_{i\in T\setminus T'}m_{i}(\alpha h^{0}(X, \mathscr{F}_{i}(k))-\alpha_{i} P(k))+a\alpha_{i}\delta))\leq \\
\leq&(\sum_{i\in T'}m_{i}(\alpha P_{ \mathscr{F}_{i}}(k)-\alpha_{i}P(k)))+\delta  \mu( \mathscr{F}'_{\bullet},\underline{m}',\phi) (\leq)0.
\end{split}
\end{equation}
Equation (\ref{canita}) implies that the inequalities in (\ref{canita2}) become equalities, so $T\setminus T'=\emptyset$ which implies $ \mathscr{F}_{i}\in S_{0}$ for all $i$.
\end{proof}

\subsubsection{Step 2: The Hilbert-Mumford criterion}

Let $\Phi\colon (F^{\otimes a})^{\oplus b}\rightarrow H^{0}(X, \mathscr{D}(ak))^{\vee}$ be a point of $\mathbf{P}(Z_{1})$ and let $(F_{\bullet},\underline{m})$ be a weighted filtration of $F$. For each $ F_{i}$ denote by $n_{i}$ its dimension  and by $n$ the dimension of $F$. Define the vector 
\begin{equation*}
\Gamma=\sum_{1}^{t}m_{i}  \Gamma^{(n_{i})},
\end{equation*}
where $\Gamma^{(l)}=(\overset{l}{\overbrace{l-n,\hdots,l-n}},\overset{n-l}{\overbrace{l,\hdots,l}})$. Let us denote by $J$ the set
\begin{equation*}
J=\{\textrm{ multi-indices }I=(i_{1},\hdots,i_{a})|I_{j}\in\{1,\hdots,t+1\}\}.
\end{equation*}
Define
$\mu(F_{\bullet},\underline{m},\Phi)= \textrm{min}_{I\in J}\{  \Gamma_{n_{i_{1}}}+\hdots +  \Gamma_{n_{i_{a}}}|\Phi|_{(F_{i_{1}}\otimes \hdots\otimes  F_{i_{a}})^{\oplus b}}\neq 0\}$. Denote by $\epsilon_{i}(F_{\bullet})$ the number
$\#\{k\in (i_{1},\hdots,i_{a})| \ n_{k}\leqslant n_{i}\}$,
$(i_{1},\hdots,i_{a})$ being a multi-index giving the minimum in $\mu(V_{\bullet},\underline{m},\Phi)$.
Then, the following holds,
\begin{equation*}
\mu(V_{\bullet},\underline{m},\Phi)=\sum_{i=1}^{t} m_{i}( \alpha_{i}a-\epsilon_{i}(F_{\bullet})\alpha).
\end{equation*}

 \begin{proposition}\label{propl}
Let $g,h,D,a,b, l_0\in\mathbb{N}$,  and $P(n)\in\mathbb{Z}[n]$ a polynomial of degree one. There is a natural number $L_0\in\mathbb{N}$ depending only on the numerical input data such that for every $l\geq L_0$ the following holds: for every Cohen-Macaulay projective and connected curve of genus $g$ with a very ample line bundle $\mathscr{O}_{X}(1)$ of degree $h$, such that $h^{1}(X,\mathscr{O}_{X}(l_0))=0$, the point $(q,\Phi)\in Z_{k,\mathscr{D}}(X)$ is $\emph{GIT}$-(semi)stable with respect to $\mathscr{O}_{Z_{k,\mathscr{D}}(X)}(n_{1},n_{2})$ if and only if for every weighted filtration $(F_{\bullet},\underline{m})$ of $F:=\mathbb{C}^{P(k)}$,
\begin{equation}\label{eqimp}
n_{1}(\sum_{1}^{t}m_{i}(\emph{dim}(F_{i})P(l)-\emph{dim}(F) P_{ \mathscr{F}_{F_{i}}}(l))) + n_{2}\delta\mu(F_{\bullet},\underline{m},\Phi)(\leq)0.
\end{equation}
Furthermore, there is an integer $A_{2}$ depending only on the numerical input data such that it is enough to consider weighted filtrations with $m_{i}\leq A_{2}$.
 \end{proposition}
\begin{proof}
It follows as in \cite[Proposition 3.4]{sols} and applying Corollary \ref{gitsemistability}. The last part follows by applying the same argument given in Remark \ref{finiteset}.
\end{proof}

 \begin{proposition}\label{colmena}
Let $g,h,D,a,b,l_0\in\mathbb{N}$, and  $P(n)\in\mathbb{Z}[n]$ a polynomial of degree one. There is a natural number $L_1\in\mathbb{N}$ depending only on the numerical input data such that for every $l\geq L_1$ the following holds: for every Cohen-Macaulay projective and connected curve of genus $g$ with a very ample line bundle $\mathscr{O}_{X}(1)$ of degree $h$, such that $h^{1}(X,\mathscr{O}_{X}(l_0))=0$,
a point $(q,\Phi)\in Z_{k,\mathscr{D}}(X)$ is \emph{GIT}-(semi)stable with respect to $\mathscr{O}_{Z_{k,\mathscr{D}}(X)}(n_{1},n_{2})$ if and only if for all weighted filtrations $( \mathscr{F}_{\bullet},\underline{m})$ of $ \mathscr{F}$,
\begin{equation*}\label{semistable2}
\sum_{1}^{t}m_{i}( (\emph{dim} (F_{ \mathscr{F}_{i}})-\epsilon_{i}(F_{\bullet})\delta)(P-a\delta)-(P_{ \mathscr{F}_{F_{i}}}-\epsilon_{i}(F_{\bullet})\delta)(\emph{dim} F-a\delta) )(\preceq)0.
\end{equation*}
where $F:=\mathbb{C}^{P(k)}$. 
 \end{proposition}
\begin{proof}
Using the polarization given in (\ref{polarization}), the inequality of Proposition \ref{propl} becomes
\begin{equation*}
\sum_{i=1}^{t}m_{i}((\textrm{dim}(F_{i})- \epsilon(F_{\bullet})\delta)(P(l)-a\delta)- (P_{\mathscr{F}_{F_{i}}}(l)-\epsilon_{i}(F_{\bullet})\delta)(\textrm{dim}(F)-a\delta))(\leq)0.
\end{equation*}
By Proposition \ref{propl} there is an $A_{2}$ such that we just need to choose $m_{i}<A_{2}$. By Corollary \ref{gitsemistability}, there exists $L'_1\in\mathbb{N}$ depending only on the numerical input data such that $h^{1}(X,\mathscr{F}_{F'}(l))=0$ for every $l\geq L'_1$. Altogether implies that there are finitely many possible Hilbert polynomials for the sheaves of the form $\mathscr{F}_{F'}$, and their coefficients are bounded by constants that do not depend on the base curve we are considering. Therefore, there is $L_1\in\mathbb{N}$, depending only on the numerical input data, such that if $l\geq L_1$, the above inequality holds if and only if it holds as inequality of polynomials. Now, the proposition follows as in \cite[Proposition 3.5]{sols}.
\end{proof}

\subsubsection{Step 3: Proof of Theorem \ref{teoremon}}
We chose natural numbers $k\geq N:=N_5$ and $l\geq L:=L_1$.

1) We will see that if $(q,\Phi)$ is GIT-(semi)stable, then the corresponding swamp, $( \mathscr{F},\phi)$, is $\delta$-(semis)stable and $q$ induces the isomorphism. From Equation (\ref{eqimp}) and the polarization defined in Equation (\ref{polarization}), we deduce that
\begin{equation*}
\sum_{1}^{t}m_{i}( (\textrm{dim} (F_{ \mathscr{F}_{i}})-\epsilon_{i}(F_{\bullet})\delta)\alpha-\alpha_{i}(\textrm{dim} (F)-a\delta) )\leq 0,
\end{equation*}
or, equivalently
\begin{equation}\label{mosquito}
\sum_{i=1}^{t}m_{i}(\textrm{dim}(F_{ \mathscr{F}_{i}})\alpha - \alpha_{i}\textrm{dim}(F))+\delta\mu( \mathscr{F}_{\bullet},\underline{m},\phi)\leq 0.
\end{equation}
Since $\textrm{dim}(F)=P(k)$ and $P(k)\leq h^{0}(X, \mathscr{F}_{i}(k))+h^{0}(X, \mathscr{F}^{i}(k))$, inequality (\ref{mosquito}) becomes
\begin{equation}\label{e}
(\sum_{i=1}^{t}m_{i}(\alpha^{i}P(k)-\alpha h^{0}(X, \mathscr{F}^{i}(k))))+\delta\mu( \mathscr{F}_{\bullet},\underline{m},\phi)\leq 0.
\end{equation}
Applying \cite[Lemma 2.3]{L} and \cite[Lemma 1.17]{simpson}, and by a similar argument as given in \cite[Theorem 3.5]{sols}, we deduce that $ \mathscr{F}$ has pure dimension one. Now, by Theorem \ref{sectional}, $( \mathscr{F},\phi)$ is $\delta$-semistable. Finally, the quotient $q$ induces a linear map
$f_{q}\colon F\rightarrow H^{0}(X, \mathscr{F}(k))$.
Let $F'\subseteq F$ be its kernel. Clearly, $ \mathscr{F}_{F'}=0$ and $\mu(F'\subseteq F,1,\Phi)=a \textrm{dim}(F')$.  Proposition \ref{propl} gives us
$n_{1}\textrm{dim}(F')P(l)+n_{2} \textrm{dim}(F')\leq 0$.
Hence $F'=0$, so $f_{q}$ is injective. Since $( \mathscr{F},\phi)$ is $\delta$-semistable, $\textrm{dim}(F)=h^{0}(X, \mathscr{F}(k))$ so $f_{q}$ is, in fact, an isomorphism.

2) For the converse of the statement, assume $( \mathscr{F},\phi)$ is $\delta$-(semi)stable and that $q$ induces an isomorphism $f_{q}:F\simeq H^{0}(X, \mathscr{F}(k))$. Since $f_{q}$ is an isomorphism, $F_{ \mathscr{F}'}=H^{0}(X, \mathscr{F}'(k))$ for any subsheaf $ \mathscr{F}'\subset \mathscr{F}$. Thus, by Theorem \ref{sectional} we have
\begin{equation}\label{leadingcoef}
\sum_{i=1}^{t}m_{i}(\alpha \textrm{dim} F_{ \mathscr{F}_{i}}-\alpha_{i}P(k))+\delta\mu( \mathscr{F}_{\bullet} ,\underline{m},\phi)(\leq) 0
\end{equation}
for all weighted filtrations. Observe that the left-hand side of Equation (\ref{leadingcoef}) is precisely the leading coefficient of the polynomial
\begin{equation*}
\sum_{i=1}^{t}m_{i}((\textrm{dim} F_{ \mathscr{F}_{i}}-\epsilon_{i}( \mathscr{F}_{\bullet})\delta )(P-a\delta)-(P_{ \mathscr{F}_{i}}-\epsilon_{i}( \mathscr{F}_{\bullet})\delta)(\textrm{dim} F-a\delta)).
\end{equation*}
We deduce that if we have a strict inequality in Equation (\ref{leadingcoef}) then,
\begin{equation*}
\sum_{i=1}^{t}m_{i}((\textrm{dim} F_{ \mathscr{F}_{i}}-\epsilon_{i}( \mathscr{F}_{\bullet})\delta )(P-s\delta)-(P_{ \mathscr{F}_{i}}-\epsilon_{i}( \mathscr{F}_{\bullet})\delta)(\textrm{dim} F-s\delta))\prec 0.
\end{equation*}
If $( \mathscr{F},\phi)$ is strictly $\delta$-semistable, by Theorem \ref{sectional} there is a filtration $( \mathscr{F}_{\bullet},\underline{m})$ giving an equality in (\ref{leadingcoef})
\begin{equation}\label{eqq1}
\sum_{i=1}^{t}m_{i}(\alpha \textrm{dim}(F_{ \mathscr{F}_{i}})-\alpha_{i} P(k))+\delta\mu( \mathscr{F}_{\bullet} ,\underline{m},\phi)= 0.
\end{equation}
Note that
\begin{align*}
\sum_{i=1}^{t}m_{i}&\left\{(\textrm{dim}(F_{ \mathscr{F}_{i}})-\epsilon_{i}\delta)(P-a\delta)-(P_{ \mathscr{F}_{i}}-\epsilon_{i}\delta)(\textrm{dim}(F)-a\delta)\right\} =\\
=\sum_{i=1}^{t}m_{i}&\left\{(\textrm{dim}(F_{ \mathscr{F}_{i}})P-\textrm{dim}(F)P_{ \mathscr{F}_{i}}) +\delta(P_{ \mathscr{F}_{i}}a-\epsilon_{i}P)\right. -\\
&\left.-\delta(\textrm{dim}(F_{ \mathscr{F}_{i}})a-\epsilon_{i}\textrm{dim}(F))  \right\}.
\end{align*}
The degree one coefficient of this polynomial is given by
\begin{align*}
&\sum_{i=1}^{t}m_{i}(  (\textrm{dim}(F_{ \mathscr{F}_{i}})\alpha-\textrm{dim}(F)\alpha_{i}) +\delta(\alpha_{i}a-\epsilon_{i}\alpha)  )=\\
=&\sum_{i=1}^{t}m_{i}(\alpha \textrm{dim}(F_{ \mathscr{F}_{i}})-\alpha_{i} P(k))+\delta\mu( \mathscr{F}_{\bullet} ,\underline{m},\phi),
\end{align*}
which is equal to $0$ because of Equation (\ref{eqq1}). Using the equalities  $P(n)=(\textrm{dim}(F)-\alpha k)+\alpha n$, $P_{ \mathscr{F}_{i}}(n)=(\textrm{dim}(F_{ \mathscr{F}_{i}})-\alpha_{i}k)+\alpha_{i}n$ (this last equality follows from Corollary \ref{strictsemi} ) and again Equation (\ref{eqq1}), it follows that the constant coefficient of this polynomial is also $0$. This implies that if $( \mathscr{F},\phi)$ is $\delta$-(semi)stable and $f_{1}$ is an isomorphism then
\begin{equation*}
\sum_{i=1}^{t}m_{i}((\textrm{dim} F_{ \mathscr{F}_{i}}-\epsilon_{i}( \mathscr{F}_{\bullet})\delta )(P-s\delta)-(P_{ \mathscr{F}_{i}}-\epsilon_{i}( \mathscr{F}_{\bullet})\delta)(\textrm{dim} F-s\delta))(\preceq) 0.
\end{equation*}
Now, the result follows from Proposition \ref{colmena}.

\section{The universal moduli space of swamps}
\label{sectionuniversal}

Our goal is to prove the existence of a coarse projective moduli space, $\mathcal{T}^{\delta\textrm{-(s)s}}_{P,g,a,b}$, for the moduli functor
\begin{equation*}
\textbf{Swamps}^{\delta\textrm{-(s)s}}_{P,g,a,b}(T)=\left\{
\begin{array}{l}
\textrm{isomorphism classes of pairs }(X_{T},( \mathscr{F}_{T},\phi_{T},\mathscr{N}))\\
\textrm{where }X_{T} \textrm{ is a semistable curve of genus }g \\
\textrm{over }T \textrm{ and } ( \mathscr{F}_{T},\phi_{T},\mathscr{N}) \textrm{ is a } \delta\textrm{-(semi)stable}\\
\textrm{swamp of uniform multi-rank }r \textrm{over } T\\
\textrm{ and with Hilbert polynomial } P
\end{array}\right\}.
\end{equation*}
satisfying that there is a natural map $\Theta_{sw}:\mathcal{T}^{\delta\textrm{-(s)s}}_{P,g,a,b}\rightarrow\overline{\textrm{M}}_{g}$ such that for any stable curve $[X]\in\overline{\textrm{M}}_{g}$,  $\Theta_{sw}^{-1}([X])=\mathcal{T}^{\delta\textrm{-(s)}s}_{P,X,a,b}/\textrm{Aut}(X)$.

\subsection{Gieseker construction of $\overline{\textrm{M}}_{g}$}
We summarize the construction of $\overline{\textrm{M}}_{g}$ \cite{gies-curves}.

Fix integers $g\geq 2$, $d=10(2g-2)$ and $M=d-g$. Consider the Hilbert scheme $\textrm{H}_{d,g,M}$ representing projective curves of genus $g$ and degree $d$ in  $\mathbf{P}^{M}$.
There exists a projective embedding,
$i'_{s}:\textrm{H}_{d,g,M}\hookrightarrow \mathbf{Grass}(h(s),H^{0}(\mathbf{P}^{M},\mathcal{O}_{\mathbf{P}^{M}}(s))^{\vee})$,
for each $s\in\mathbb{N}$ greater than certain $s_{s}\in\mathbb{N}$. 
Given a stable curve $X$ of genus $g$, $\mathscr{\omega}_{X}^{\otimes 10}$ is a very ample line bundle with $\textrm{dim}(H^{0}(X,\mathscr{\omega}_{X}^{\otimes 10}))=M+1$. Thus, once an isomorphism $\mathbb{C}^{M+1}\simeq H^{0}(X,\mathscr{\omega}_{X}^{\otimes 10})$ is fixed, $\mathscr{\omega}_{X}^{\otimes 10}$ embeds $X$ in the projective space $\mathbf{P}^{M}$, the image being a projective curve of genus $g$ and degree $d$. Therefore, 
$[X]\in\textrm{H}_{d,g,M}$.

Let us denote by   $\textrm{H}_{g}\subset \textrm{H}_{d,g,M}$ the locus of non-degenerate, 10-canonical stable curves of genus $g$. 
The action of $\textrm{SL}_{M+1}$ on $\mathbf{P}^{M}$ induces an action on $\textrm{H}_{g}$. Gieseker shows that there exist a natural number $s_{1}$ such that for every $s\geq s_1$ the GIT linearised problem $i'_{s}$ satisfies
(1) $\textrm{H}_{g}$ belongs to the semistable locus,
(2) $\textrm{H}_{g}$ is closed in the semistable locus,
from what follows that
$
\overline{\textrm{M}}_{g}=\textrm{H}_{g}/\textrm{SL}_{M+1}
$
exists and is projective.
The scheme $\textrm{H}_{g}$ is endowed with a universal family 
\begin{equation}\label{uno}
\xymatrix{
U_{g}\ar@{^(->}[rr]^{\small{closed}}_{\psi}\ar[rrd]^{\mu}_{flat} & & \textrm{H}_{g}\times\mathbf{P}^{M} \ar[r]^{pr_{2}}\ar[d]^{pr_{1}} & \mathbf{P}^{M} \\
&& \textrm{H}_{g} &
}
\end{equation}
called the universal curve of genus $g$. We will denote by $\nu:U_{g}\rightarrow\mathbf{P}^{M}$ the second projection. 

The projective scheme $\textrm{H}_{g}$ has the following important feature. For any closed point $h\in\textrm{H}_{g}$, $\psi$ induces a closed immersion $\psi_{h}:X_{h}\hookrightarrow \mathbf{P}^{M}$, $X_{h}$ being the fiber of $\mu$ over $h\in\textrm{H}_{g}$, which satisfies that $\psi_{h}^{*}\mathscr{O}_{\mathbf{P}^{M}}(1)\simeq \mathscr{\omega}_{X_{h}}^{\otimes 10}$ (see \cite[Proposition 2.0.0]{gies-curves}). Therefore, $\nu^{*}\mathscr{O}_{\mathbf{P}^{M}}(1)|_{X_{h}}\simeq \mathscr{\omega}_{X_{h}}^{\otimes 10}$ for every $h\in\textrm{H}_{g}$.

\subsection{Grothendieck embedding of the relative Quot scheme}\label{embedquot}

Consider the relatively very ample line bundle $\nu^{*}\mathscr{O}_{\mathbf{P}^{M}}(1)=:\mathscr{O}_{U_g}(1)$ on the universal curve. Define
$\mathbf{Q}^{r}_{g}(\mu,k,P):=\textrm{Quot}^{P,r}_{\mathbb{C}^{n}\otimes\mathscr{O}_{U_{g}}(-k)/U_{g}/ \textrm{H}_{g}}\subset\textrm{Quot}^{P}_{\mathbb{C}^{n}\otimes\mathscr{O}_{U_{g}}(-k)/U_{g}/ \textrm{H}_{g}}$,
where $k\in\mathbb{N}$ and $n=P(k)$. This is the (open and closed) subscheme of quotients with uniform multi-rank $r$.   
By construction, there is a canonical projective morphism
\begin{equation}\label{dos}
\mathbf{Q}^{r}_{g}(\mu,k,P)\overset{\pi}{\rightarrow}\textrm{H}_{g}
\end{equation}
and the fibered product
$$
\xymatrix{
\mathbf{Q}^{r}_{g}(\mu,k,P)\times_{\textrm{H}_{g}}U_{g}\ar[r]^{\theta}\ar[d]^{\phi} & \mathbf{Q}^{r}_{g}(\mu,k,P)\ar[d]^{\pi} \\
U_{g}\ar[r]_{\mu} & \textrm{H}_{g}
}
$$
is equipped with a universal  quotient
\begin{equation}\label{tres}
q_{U_g}:\mathbb{C}^{n}\otimes\phi^{*}\mathscr{O}_{U_{g}}(-k)\twoheadrightarrow \mathscr{E}\rightarrow 0
\end{equation}
flat over $\mathbf{Q}_{g}(\mu,k,f)$ (see \cite{FGA}).
From Equations (\ref{uno}) and (\ref{dos}) we find a closed immersion
\begin{equation*}
\mathbf{Q}^{r}_{g}(\mu,k,P)\times_{\textrm{H}_{g}}U_{g}\overset{id\times\psi}{\hookrightarrow}\mathbf{Q}^{r}_{g}(\mu,k,P)\times_{\textrm{H}_{g}}(\textrm{H}_{g}\times\mathbf{P}^{M})\simeq \mathbf{Q}^{r}_{g}(\mu,k,P)\times\mathbf{P}^{M}
\end{equation*}
We can push the universal quotient (\ref{tres}) forward to $ \mathbf{Q}^{r}_{g}(\mu,k,P)\times\mathbf{P}^{M}$ and compose it with the natural surjection  $\mathbb{C}^{n}\otimes \mathscr{O}_{\mathbf{Q}^{r}_{g}(\mu,k,P)\times\mathbf{P}^{M}}(-k)\twoheadrightarrow \mathbb{C}^{n}\otimes(id\times\psi)_{*}\phi^{*}\mathscr{O}_{U_{g}}(-k)$. This, in turn, induces an exact sequence,
\begin{equation*}
0\rightarrow \K\hookrightarrow  \mathbb{C}^{n}\otimes \mathscr{O}_{\mathbf{Q}^{r}_{g}(\mu,k,P)\times\mathbf{P}^{M}}(-k)\twoheadrightarrow\overline{ \mathscr{E}}\rightarrow 0,
\end{equation*}
all the sheaves being flat over $\mathbf{Q}^{r}_{g}(\mu,k,P)$. Further, there exists an integer $l_{1}$ such that, for every $l\geq l_1$, there is a projective embedding
\begin{align*}
i_{l}: \mathbf{Q}^{r}_{g}(\mu,k,P)\rightarrow \mathbf{Grass}(P(l), \mathbb{C}^{n}\otimes H^{0}(\mathbf{P}^{M},\mathcal{O}_{\mathbf{P}^{M}}(l-k))). 
\end{align*}
There is, in fact, an integer $l_{2}$ such that $ \forall l>l_{2}$($l>k$) there is a closed immersion
\begin{equation}
\pi\times i_{l}: \mathbf{Q}^{r}_{g}(\mu,k,P)\hookrightarrow \textrm{H}_{g}\times \mathbf{Grass}(P(l),\mathbb{C}^{n}\otimes H^{0}(\mathbf{P}^{M},\mathcal{O}_{\mathbf{P}^{M}}(l-k))),
\end{equation}
and composing with the Pl\"ucker embedding we get the desired projective embedding,
\begin{equation}
\pi\times i_{l}: \mathbf{Q}^{r}_{g}(\mu,k,P)\hookrightarrow \textrm{H}_{g}\times \mathbf{P}(H_3),
\end{equation}
with $H_3:=\bigwedge^{P(l)}( \mathbb{C}^{n}\otimes H^{0}(\mathbf{P}^{M},\mathscr{O}_{\mathbf{P}^{M}}(l-k)))$.

\subsection{Projective embedding of swamps data}\label{embedswamp}

The results proved so far show that there exists $k\in\mathbb{N}$ large enough such that any pair $(X,( \mathscr{F},\phi))$, given by a stable curve of genus $g$ (with polarization $\mathscr{O}_{X}(1):=\mathscr{\omega}_{X}^{\otimes 10}$) and a $\delta$-semistable swamp with Hilbert polynomial $P$, defines a point in the relatively projective $\textrm{H}_{g}$-scheme
$$
\mathbf{Q}^{r}_{g}(\mu,k,P)\times \mathbf{P}((((\mathbb{C}^{n})^{\otimes a})^{\oplus b})^{\vee}\otimes\mu_{*}\mathscr{O}_{U_g}(ak))\rightarrow \textrm{H}_{g}
$$

The natural number $k\in\mathbb{N}$ is fixed as before. Since $h^{1}(X,\mathscr{\omega}_{X}^{\otimes 10})=0$ and $h^{0}(X,\mathscr{\omega}_{X}^{\otimes 10})=10(2g-2)-g+1$ for every stable curve, we have that $\mu_{*}\mathscr{O}_{U_g}(ak)$ is locally free, and that  $R^{1}\mu_{*}\mathscr{O}_{U_g}(ak)=0$. Consider now the projective bundle
$$
\overline{\pi}: \mathbf{P}((((\mathbb{C}^{n})^{\otimes a})^{\oplus b})^{\vee}\otimes\mu_{*}\mathscr{O}_{U_g}(ak))\rightarrow  \textrm{H}_{g}
$$
Since $\mu_{*}\mathscr{O}_{U_g}(ak)=pr_{1*}\psi_{*}(\psi^{*}pr_{2}^{*}\mathscr{O}_{\mathbf{P}^{M}}(ak))$, the natural surjection
\begin{equation}\label{sur}
pr_{2}^{*}\mathscr{O}_{\mathbf{P}^{M}}(ak)\twoheadrightarrow\psi_{*} \psi^{*}pr_{2}^{*}\mathscr{O}_{\mathbf{P}^{M}}(ak),
\end{equation}
induces a diagram
\begin{equation*}
\xymatrix{
pr_{1*}pr_{2}^{*}\mathscr{O}_{\mathbf{P}^{M}}(ak)\eq[d]\ar[r]^{v_{k}} & pr_{1*}\psi_{*}\psi^{*}pr_{2}^{*} \mathscr{O}_{\mathbf{P}^{M}}(ak)\ar@{=}[d] \\
\mathscr{O}_{H_{g}}\otimes_{\mathbb{C}} H^{0}(\mathbf{P}^{M},\mathscr{O}_{\mathbf{P}^{M}}(ak)) & \mu_{*}\mathscr{O}_{U_g}(ak)
}
\end{equation*}
and therefore a morphism $\mathscr{O}_{H_{g}}\otimes_{\mathbb{C}} H^{0}(\mathbf{P}^{M},\mathscr{O}_{\mathbf{P}^{M}}(ak)) \rightarrow  \mu_{*}\mathscr{O}_{U_g}(ak)$, which will be denoted by $v_{k}'$.
By Serre's theorem, there is a natural number $N'\in\mathbb{N}$, that may be taken grater than $N$, such that if $k\geq N'$ then $v_{k}$ is surjective and we have a closed immersion
$$
\xymatrix{
\mathbf{P}((((\mathbb{C}^{n})^{\otimes a})^{\oplus b})^{\vee}\otimes\mu_{*}\mathscr{O}_{U_g}(ak))\ar@{^(->}[rr]_{\hspace{1.2cm}\tiny{closed}} \ar[rd]_{\overline{\pi}}& & \textrm{H}_{g}\times \mathbf{P}(H_1)\ar[ld] \\
  & \textrm{H}_{g} &
}
$$
where $H_{1}=(((\mathbb{C}^{n})^{\otimes a})^{\oplus b})^{\vee}) \otimes H^{0}(\mathbf{P}^{M},\mathscr{O}_{\mathbf{P}^{M}}(ak))$.

\subsection{Parameter space for swamps and the polarization}\label{parameterswamps}

Consider the fibered product
{ \footnotesize
\begin{equation*}
\xymatrix{
Y:=Q_{g}(\mu,n,f)\times_{H_{g}}\mathbf{P}((((\mathbb{C}^{n})^{\otimes a})^{\oplus b})^{\vee}\otimes\mu_{*}\mathscr{O}_{U_g}(ak))\ar[r]_{\ \ \ \ \ \ \ \ \ \ \pi_{2}}\ar[d]_{\pi_{1}}\ar@{-->}[rd]_{w} & \mathbf{P}((((\mathbb{C}^{n})^{\otimes a})^{\oplus b})^{\vee}\otimes\mu_{*}\mathscr{O}_{U_g}(ak)) \ar[d]^{\overline{\pi}} \\
\mathbf{Q}^{r}_{g}(\mu,k,P)\ar[r]^{\pi} & H_{g} 
}
\end{equation*}}
Giving $\pi_{2}$ is the same as giving a quotient invertible sheaf
\begin{equation}
(((\mathbb{C}^{n})^{\otimes a})^{\oplus b})^{\vee}\otimes w^{*}\mu_{*}\mathscr{O}(ak)\twoheadrightarrow \mathscr{L}
\end{equation}
on $Y$, which is the same as giving a nonzero morphism
\begin{equation}\label{cinco}
\varphi'_{Y}:(((\mathbb{C}^{n})^{\otimes a})^{\oplus b})\otimes \mathscr{O}_{Y} \twoheadrightarrow w^{*}\mu_{*}\mathscr{O}_{U_g}(ak)\otimes \mathscr{L},
\end{equation}
while giving $\pi_{1}$ is the same as giving a quotient sheaf
\begin{equation}\label{seis}
q_{Y}:\mathbb{C}^{n}\otimes \mathscr{O}_{Y\times_{H_{g}}U_{g}}\twoheadrightarrow  \mathscr{E}(k)
\end{equation}
on $Y\times_{H_{g}}U_{g}$. Now, we can pull (\ref{cinco}) back to $Y\times_{H_{g}}U_{g}$, and we get
\begin{equation}\label{siete}
\varphi''_{Y}:(((\mathbb{C}^{n})^{\otimes a})^{\oplus b})\otimes \mathscr{O}_{Y\times_{H_{g}}U_{g}} \twoheadrightarrow \pi_{Y}^{*}w^{*}\mu_{*}\mathscr{O}_{U_g}(ak)\otimes \pi_{Y}^{*}\mathscr{L}.
\end{equation}
From Equation (\ref{seis}) and Equation (\ref{siete}) we can form the following diagram
$$
\xymatrix{
0\ar[r]  & \mathscr{K}\ar@{^(->}[r]\ar@/_10mm/[rddd]_{\overline{ \overline{\varphi_{Y}}}} & (((\mathbb{C}^{n})^{\otimes a})^{\oplus b})\otimes \mathscr{O}_{Y\times_{H_{g}}U_{g}}\ar[r] \ar[d]^{\varphi''_{Y}}\ar@/^25mm/[ddd]^ {\overline{\varphi_{Y}}} & ( \mathscr{E}(k)^{\otimes a})^{\oplus b}\ar[r] & 0\\
&&\pi_{Y}^{*}w^{*}\mu_{*}\mathscr{O}_{U_g}(ak)\otimes \pi_{Y}^{*}\mathscr{L}\ar@{=}[d]& & \\
&& \pi_{U_{g}}^{*}\mu^{*}\mu_{*}\mathscr{O}_{U_g}(ak) \otimes\pi_{Y}^{*}\mathscr{L}\ar@{->>}[d] && \\
&& \pi_{U_{g}}^{*}\mathscr{O}_{U_g}(ak)\otimes \pi_{Y}^{*}\mathscr{L} &&
}
$$
Since $\pi_{U_{g}}^{*}\mathscr{O}_{U_g}(ak)\otimes\pi_{Y}^{*}\mathscr{L}$ is flat over $Y$, there exists a closed subscheme
\begin{equation*}
\mathcal{Z}\subset \mathbf{Q}^{r}_{g}(\mu,k,P)\times_{H_{g}} \mathbf{P}((((\mathbb{C}^{n})^{\otimes a})^{\oplus b})^{\vee}) \otimes\mu_{*}\mathscr{O}_{U_g}(ak))
\end{equation*}
characterized by the fact that $\overline{\overline{\varphi_{Y}}}|_{\mathcal{Z}}=0$. Restricting the diagram to $\mathcal{Z}$ we show that $\overline{\varphi_{Y}}$ lifts to $( \mathscr{E}(k)^{\otimes a})^{\oplus b}$, that is, it factorices through a morphism
\begin{equation}\label{ocho}
\varphi_{\mathcal{Z}}:( \mathscr{E}(k)|_{\mathcal{Z}}^{\otimes a})^{\oplus b}\rightarrow \pi_{U_{g}}^{*}\mathscr{O}_{U_g}(ak)\otimes\pi_{Y}^{*}\mathscr{L}.
\end{equation}
Then, the closed subscheme $\mathcal{Z}\subset Y$ carries a universal family of swamps over $\textrm{H}_{g}$
\begin{equation}\label{unifamilyswamps}
\begin{split}
q_{\mathcal{Z}}&:\mathbb{C}^{n}\otimes \mathscr{O}_{\mathcal{Z}\times_{H_{g}}U_{g}}(-k) \rightarrow  \mathscr{E}|_{\mathcal{Z}}, \\
\varphi_{\mathcal{Z}}&:( \mathscr{E}(k)|_{\mathcal{Z}}^{\otimes a})^{\oplus b}\rightarrow \pi_{U_{g}}^{*}\mathscr{O}_{U_g}(ak)\otimes\pi_{Y}^{*}\mathscr{L}.
\end{split}
\end{equation}

The group $\textrm{SL}_{M+1}$, acting on $\textrm{H}_{g}$, induces naturally actions on $\mathbf{Q}^{r}_{g}(\mu,k,P)$ and $ \mathbf{P}((((\mathbb{C}^{n})^{\otimes a})^{\oplus b})^{\vee}\otimes\mu_{*}\mathscr{O}_{U_g}(ak))$, and the group $\textrm{SL}_{n}$, acting on $\mathbb{C}^{n}$, induces actions on $\mathbf{Q}^{r}_{g}(\mu,k,P)$ and $ \mathbf{P}((((\mathbb{C}^{n})^{\otimes a})^{\oplus b})^{\vee}\otimes\mu_{*}\mathscr{O}_{U_g}(ak))$ as well. Therefore, the group $\textrm{SL}_{M+1}\times \textrm{SL}_{n}$ is acting on $\mathbf{Q}^{r}_{g}(\mu,k,P)\times \mathbf{P}((((\mathbb{C}^{n})^{\otimes a})^{\oplus b})^{\vee}\otimes\mu_{*}\mathscr{O}_{U_g}(ak))$ and the actions commute with each other.

Combining Section \ref{embedquot} and Section \ref{embedswamp},  we find a projective embedding
$$
\xymatrix{
\mathcal{Z}\ar@{^(->}[d] \\
\mathbf{Q}^{r}_{g}(\mu,k,P)\times_{H_{g}}\mathbf{P}((((\mathbb{C}^{n})^{\otimes a})^{\oplus b})^{\vee}) \otimes\mu_{*}\mathscr{O}_{U_g}(ak))\ar@{^(->}[d] \\
 H_{g}\times \mathbf{Grass}(P(l),\mathbb{C}^{n}\otimes H^{0}(\mathbf{P}^{M},\mathscr{O}_{\mathbf{P}^{M}} (l-k)))\times
\mathbf{P}(H_1)
}
$$
Finally we can conclude that for large $k,l,s$ we have a closed immersion
\begin{equation*}
\begin{split}
\mathcal{Z}\hookrightarrow & \mathbf{Grass}(h(s),H^{0}(\mathbf{P}^{M},\mathscr{O}_{\mathbf{P}^{M}}(s)))\times\\ 
\times &\mathbf{Grass}(P(l),\mathbb{C}^{n}\otimes H^{0}(\mathbf{P}^{M},\mathscr{O}_{\mathbf{P}^{M}} (l-k))) \times\\
\times & \mathbf{P}(H_1) 
\end{split}
\end{equation*}
Considering the Pl\"ucker embedding, we finally get the closed immersion
$$
\xymatrix{
j_{s,t,k}:\mathcal{Z}\ar@{^(->}[r] & \mathbf{P}(H_1)\times\mathbf{P}(H_2)\times\mathbf{P}(H_3),
}
$$
where
\begin{align*}
H_1=&(((\mathbb{C}^{n})^{\otimes a})^{\oplus b})^{\vee} \otimes H^{0}(\mathbf{P}^{M}, \mathscr{O}_{\mathbf{P}^{M}}(ak)) \\
H_2=&\bigwedge^{h(s)} H^{0}(\mathbf{P}^{M}, \mathscr{O}_{\mathbf{P}^{M}}(\overline{s})) \\
H_3=&\bigwedge^{P(l)}( \mathbb{C}^{n}\otimes H^{0}(\mathbf{P}^{M},\mathscr{O}_{\mathbf{P}^{M}}(l-k))),
\end{align*}
and which is $\textrm{SL}_{M+1}\times\textrm{SL}_{n}$-equivariant. For each $i=1,2,3$, denote by $\pi_{i}$ the projection onto the $i$th factor, $\pi_{i}:\mathbf{P}(H_1)\times\mathbf{P}(H_2)\times\mathbf{P}(H_3)\rightarrow \mathbf{P}(H_i)$.
Let $\pi_{1}^{*}\mathscr{O}(\alpha)\otimes \pi_{3}^{*}\mathscr{O}(\gamma)$ be a polarization on $\mathbf{P}(H_1)\times \mathbf{P}(H_3)$ and recall that 
\begin{equation}
\begin{split}
H^{0}(\mathbf{P}(H_1),\mathscr{O}_{\mathbf{P}(H_1)}(\alpha)) =S^{\alpha}H_1, \ H^{0}(\mathbf{P}(H_3),\mathscr{O}_{\mathbf{P}(H_3)}(\gamma))=S^{\gamma}H_3
\end{split}
\end{equation}
We have a canonical surjection
\begin{equation*}
(S^{\alpha}H_1\otimes S^{\gamma}H_3)\otimes \mathscr{O}_{\mathbf{P}(H_1)\times \mathbf{P}(H_3)}\twoheadrightarrow \pi_{1}^{*}\mathscr{O}_{\mathbf{P}(H_1)}(\alpha) \otimes \pi_{3}^{*}\mathscr{O}_{\mathbf{P}(H_3)}(\gamma)
\end{equation*}
and then a canonical morphism, $s:\mathbf{P}(H_1)\times\mathbf{P}(H_3)\rightarrow \mathbf{P}(S^{\alpha}H_{1}\otimes S^{\gamma}H_3)$, which is in fact a closed immersion (the $(\alpha,\gamma)$ Segre embedding). Define $J:=S^{\alpha}H_1\otimes S^{\gamma}H_3$. Clearly $s^{*}\mathscr{O}_{\mathbf{^P}(J)}(1)= \pi_{1}^{*}\mathscr{O}_{\mathbf{P}(H_1)}(\alpha)\otimes\pi_{3}^{*} \mathscr{O}_{\mathbf{P}(H_3)}(\gamma)$. Consider now the composition:
\begin{equation}\label{impimmersion}
j_{s,l,k}:\mathcal{Z}\hookrightarrow \mathbf{P}(H_1)\times \mathbf{P}(H_2) \times \mathbf{P}(H_3)\overset{\pi_{2}\times s}{\longhookrightarrow} \mathbf{P}(H_2) \times \mathbf{P}(J).
\end{equation}
For the polarization $\mathscr{O}(\beta,1)$ on $\mathbf{P}(H_2)\times \mathbf{P}(J)$, with $\beta\in\mathbb{N}$, we have
\begin{equation*}
(\pi_{2}\times s)^{*}\mathscr{O}(\beta,1)= \pi_{1}^{*}\mathscr{O}_{\mathbf{P}(H_1)}(\alpha)\otimes \pi_{2}^{*}\mathscr{O}_{\mathbf{P}(H_2)}(\beta)\otimes\pi_{3}^{*}\mathscr{O}_{\mathbf{P}(H_3)}(\gamma).
\end{equation*}
From now onwards, $\alpha$ and $\gamma$ are assumed to satisfy 
\begin{equation}
\dfrac{\alpha}{\gamma}=\dfrac{P(l)-P(k)}{P(k)-a\delta} \delta,
\end{equation}
as in the fiber-wise problem.

\subsection{Construction of the universal moduli space}
Let $\xi: \textrm{SL}_{M+1}\rightarrow \textrm{SL}(H_2)$ and $\omega: \textrm{SL}_{M+1}\rightarrow \textrm{SL}(J)$ be two rational representations of $\textrm{SL}_{M+1}$. Denote by
\begin{align*}
\rho_{H_2}:& \mathbf{P}(H_2)\times \mathbf{P}(J)\rightarrow \mathbf{P}(H_2),\\
\overline{\rho_{H_2}}:& \mathbf{P}(H_1)\times \mathbf{P}(H_2) \times \mathbf{P}(H_3)\rightarrow \mathbf{P}(H_2),
\end{align*}
the projections onto $\mathbf{P}(H_2)$. The following two  results  will be applied to $\mathbf{P}(H_2)\times \mathbf{P}(J)$.
 \begin{proposition}\emph{\cite[Proposition 7.1.1]{Pando}}
There exists $\beta_0=\beta_{0}(\xi,\omega)$ such that $ \forall \beta>\beta_0$,
\begin{equation*}
\rho_{H_2}^{-1}(\mathbf{P}(H_2)^{s})\subset (\mathbf{P}(H_2)\times \mathbf{P}(J))^{s}_{[\beta,1]}.
\end{equation*}
 \end{proposition}
 \begin{proposition}\emph{\cite[Proposition 7.1.2]{Pando}}
There exists $\beta_1=\beta_{1}(\xi,\omega)$ such that $ \forall \beta>\beta_{1}$,
\begin{equation*}
(\mathbf{P}(H_2)\times \mathbf{P}(J))^{ss}_{[\beta,1]}\subset \rho_{H_2}^{-1}(\mathbf{P}(H_2)^{ss}).
\end{equation*}
 \end{proposition}
The superscripts $\{s',ss'\},\{s'',ss''\},\{s,ss\}$ will denote stability (semistability) with respect $\textrm{SL}_{M+1},\textrm{SL}_{n}$ and $\textrm{SL}_{M+1}\times \textrm{SL}_{n}$ respectively, while the subscripts $[\cdot,  \cdot]$ (or $[\cdot,\cdot,\cdot]$) denotes the polarization the semistability condition is being analyzed respect with. Define $\mathcal{Z}^{ss}_{[\alpha,\beta,\gamma]}:=\mathcal{Z}\cap (\mathbf{P}(V)\times \mathbf{P}(W)\times \mathbf{P}(H))^{ss}_{[\alpha,\beta,\gamma]}$.

 \begin{proposition}\label{existenceswamps}
The quotient  $\mathcal{T}^{\delta\textrm{-(s)s}}_{P,g,a,b}:=\mathcal{Z}^{ss}_{[\alpha,\beta,\gamma]}\slas (\emph{SL}_{N+1}\times \emph{SL}_{n})$ exists and is projective.
 \end{proposition}
\begin{proof}
By above propositions we can find $\beta>\textrm{max}\{\beta_{0},\beta_{1}\}$ such that
\begin{align}\label{pandos1}
\rho_{H_2}^{-1}(\mathbf{P}(H_2)^{s'})&\subset (\mathbf{P}(H_2)\times \mathbf{P}(J))^{s'}_{[\beta,1]}\\ \label{pandos2}
(\mathbf{P}(H_2)\times \mathbf{P}(J))^{ss'}_{[\beta,1]}&\subset \rho_{H_2}^{-1}(\mathbf{P}(H_2)^{ss'})
\end{align}
If we apply $(id\times s)^{-1}$ to Equation (\ref{pandos1}) and Equation (\ref{pandos2}) we get
{\small
\begin{equation}\label{igual}
\xymatrix{
(id\times s)^{-1}\rho_{H_2}^{-1}(\mathbf{P}(H_2)^{s'})\ar@{=}[d] & \subset &(id\times s)^{-1} (\mathbf{P}(H_2)\times \mathbf{P}(J))^{s'}_{[\beta,1]}\ar@{=}[d] \\
\overline{\rho_{H_2}}^{-1}(\mathbf{P}(H_2)^{s'}) & & (\mathbf{P}(H_1)\times \mathbf{P}(H_2)\times \mathbf{P}(H_3))^{s'}_{[\alpha,k,\beta]} \\ 
(id\times s)^{-1}(\mathbf{P}(H_2)\times \mathbf{P}(J))^{ss'}_{[\beta,1]}\ar@{=}[d]&\subset &(id\times s)^{-1}\rho_{H_2}^{-1}(\mathbf{P}(H_2)^{ss'}) \ar@{=}[d]\\
(\mathbf{P}(H_1)\times \mathbf{P}(H_2)\times \mathbf{P}(H_3))^{ss'}_{[\alpha,\beta,\gamma]} & & \overline{\rho_{H_2}}^{-1}(\mathbf{P}(H_2)^{ss'})
}
\end{equation}
}
From the fact  that $H_{g}\subset \mathbf{P}(H_2)^{s'}$ and Equation (\ref{igual}), we find
\begin{equation}\label{inclu1}
H_{g}\times \mathbf{P}(H_1)\times \mathbf{P}(H_3) \subset (\mathbf{P}(H_1)\times \mathbf{P}(H_2)\times \mathbf{P}(H_3))^{s'}_{[\alpha,\beta,\gamma]}
\end{equation}
Therefore we have 
\begin{equation}\label{inclu2}
\mathcal{Z}\subset (\mathbf{P}(H_1)\times \mathbf{P}(H_2)\times \mathbf{P}(H_3))^{s'}_{[\alpha,\beta,\gamma]}
\end{equation}
Since $\textrm{H}_{g}$ is closed in $\mathbf{P}(H_2)^{ss'}$ we know that
\begin{equation}
\textrm{H}_{g}\times \mathbf{P}(H_1)\times \mathbf{P}(H_3)\underset{closed}{\subset}\rho_{H_2}^{-1} (\mathbf{P}(H_2)^{ss'})
\end{equation}
Now, we have the following diagram
{ \footnotesize
\begin{equation*}
\xymatrix@C=3mm{
\mathcal{Z}\ar@{^(->}[r]\ar@{^(->}[d]_{\varepsilon} & (\mathbf{P}(H_1)\times \mathbf{P}(H_2)\times \mathbf{P}(H_3))^{s'}_{[\alpha\beta,\gamma]} \ar@{^(->}[r] & (\mathbf{P}(H_1)\times \mathbf{P}(H_2)\times \mathbf{P}(H_3))^{ss'}_{[\alpha,\beta,\gamma]}\ar@{^(->}[d]\\
H_{g}\times \mathbf{P}(H_1)\times \mathbf{P}(H_3) \ar@{^(->}[rr]^{closed} & & \overline{\rho_{H_2}}^{-1} (\mathbf{P}(H_2)^{ss'})
}
\end{equation*}}
Since $\mathcal{Z}$ is projective over $\textrm{H}_{g}$, $\varepsilon$ is closed, so 
$$\mathcal{Z}\subset (\mathbf{P}(H_1)\times \mathbf{P}(H_2)\times \mathbf{P}(H_3))^{ss'}_{[\alpha,\beta,\gamma]}\overset{\varsigma}{\subset} \overline{\rho_{H_2}}^{-1}(\mathbf{P}(H_2)^{ss'})$$ 
is closed in $\overline{\rho}_{H_2}(\mathbf{P}(H_2)^{ss'})$. Since $\varsigma$ is open and both spaces are open we conclude that $\mathcal{Z}$ is closed in $(\mathbf{P}(H_1)\times \mathbf{P}(H_2)\times \mathbf{P}(H_3))^{ss'}_{[\alpha,\beta,\gamma]}$.
Since 
$$(\mathbf{P}(H_1)\times \mathbf{P}(H_2)\times \mathbf{P}(H_3))^{ss}_{[\alpha,\beta,\gamma]}\subset(\mathbf{P}(H_1)\times \mathbf{P}(H_2)\times \mathbf{P}(H_3))^{ss'}_{[\alpha,\beta,\gamma]},$$ 
and
$$\mathcal{Z}^{ss}_{[\alpha,\beta,\gamma]}=\mathcal{Z}\cap (\mathbf{P}(V)\times \mathbf{P}(W)\times \mathbf{P}(H))^{ss}_{[\alpha,\beta,\gamma]},$$ 
we deduce that $\mathcal{Z}^{ss}_{[\alpha,\beta,\gamma]}$ is closed in $(\mathbf{P}(V)\times \mathbf{P}(W)\times \mathbf{P}(H))^{ss}_{[\alpha,\beta,\gamma]}$.
Therefore,
\begin{equation}
\mathcal{T}^{\delta\textrm{-(s)s}}_{P,g,a,b}:=\mathcal{Z}^{s(s)}_{[\alpha,\beta,\gamma]}\slas (\textrm{SL}_{M+1}\times \textrm{SL}_{n})
\end{equation}
exists and is projective.
\end{proof}

 \begin{proposition}\label{closedpoints}
The points of $\mathcal{Z}^{s(s)}_{[k,\alpha,\beta]}$ correspond exactly with pairs $(X,(q,\phi))$ where $X$ is a 10-canonical stable curve of genus $g$, $q:\mathbb{C}^{n}\otimes\mathscr{O}_{X}(-k)\rightarrow \mathscr{F}$ is a quotient of uniform multi-rank $r$ such that $H^{0}(q(k)):\mathbb{C}^{n}\rightarrow H^{0}(X, \mathscr{F}(k))$ is an isomorphism, $( \mathscr{F},\phi)$ being a $\delta$-(semi)stable swamp of type $(a,b,\mathscr{O}_{X})$ and with Hilbert polynomial $P(n)$.
 \end{proposition}
\begin{proof}
If we prove that there are equalities
\begin{equation}
\mathcal{Z}^{s(s)}_{[k,\alpha,\beta]}=\mathcal{Z}^{s''(s'')}_{[k,\alpha,\beta]},
\end{equation}
then the result will follow from Theorem \ref{teoremon} and the construction of $\mathcal{Z}^{ss}_{[\alpha,\beta,\gamma]}$. To prove the equalities we use the Hilbert-Mumford criterium. Consider the three representations
\begin{align*}
\xi^{\beta}:& \textrm{SL}_{M+1}\times \textrm{SL}_{n}\rightarrow \textrm{SL}_{M+1} \rightarrow \textrm{SL}(S^{\beta}(H_2))\\
\omega_{1}:& \textrm{SL}_{M+1}\times \textrm{SL}_{n}\rightarrow \textrm{SL}(S^{\alpha}(H_1)) \\
\omega_{2}:& \textrm{SL}_{M+1}\times \textrm{SL}_{n}\rightarrow \textrm{SL}(S^{\gamma}(H_3))
\end{align*}
Note that we obviously have the inclusion $\mathcal{Z}_{[\alpha,\beta,\gamma]}^{s(s)}\subset \mathcal{Z}_{[\alpha,\beta,\gamma]}^{s''(s'')}$. The representations $\omega_{1}$ and $\omega_{2}$ give us a representation
\begin{equation*}
\omega=\omega_{1}\otimes\omega_{2}: \textrm{SL}_{M+1}\times \textrm{SL}_{n}\rightarrow \textrm{SL}(J)
\end{equation*}
Let $\xi\in\mathcal{Z}_{[\alpha,\beta,\gamma]}^{s''(s'')}$ and define $(\xi_{1},\xi_{2}):=j_{s,l,k}(\xi)$, $j_{s,l,k}$ being the closed immersion given in Equation (\ref{impimmersion}). Let $\lambda:\mathbb{G}_{m}\rightarrow \textrm{SL}_{M+1}\times \textrm{SL}_{n}$ be a nontrivial 1-PS given by
\begin{align*}
\lambda_{1}:\mathbb{G}_{m}\rightarrow \textrm{SL}_{M+1}, \  \lambda_{2}:\mathbb{G}_{m}\rightarrow \textrm{SL}_{n}
\end{align*}
Now the result follows by a similar argument as the one given in \cite[Proposition 2.8.1]{Pando}. 
\end{proof}
 \begin{theorem}\label{Mainteo1}
The projective scheme $\mathcal{T}^{\delta\textrm{-(s)s}}_{P,g,a,b}$ is a coarse moduli space for the moduli functor $\emph{\textbf{Swamps}}^{\delta\textrm{-(s)s}}_{P,g,a,b}$.
 \end{theorem}
\begin{proof}
Let $T$ be a scheme and $\eta\in\textrm{\textbf{Swamps}}^{\delta\textrm{-(s)s}}_{P,g,a,b}(T)$, which consists on
\begin{equation*}
\eta=\left\{  \begin{array}{l}  
\pi:X_{T}\rightarrow T \textrm{ a flat family of stable curves of genus }g \\
\textrm{with relative polarization }\mathscr{O}_{X_{T}}(1):=\mathscr{\omega}_{X_{T}/T}^{\otimes 10}\\
 \mathscr{F}_{T} \textrm{ a flat family of coherent sheaves of pure dimension one with Hilbert}\\
P(n)\textrm{ polynomial  and uniform multi-rank }r,\textrm{ and }\mathscr{N} \textrm{ an invertible sheaf on }T\\
\phi:( \mathscr{F}_{T}^{\otimes a})^{\oplus b}\rightarrow \pi^{*}\mathscr{N} \textrm{ a }\delta\textrm{-(semi)stable swamp}\\
\end{array} 
\right.
\end{equation*}
Let $k\geq N'$ be the natural number fixed for the construction of $\mathcal{T}^{\delta\textrm{-(s)s}}_{P,g,a,b}$ (se Section \ref{embedswamp}). Then $R^{1}\pi_{*} \mathscr{F}_{T}(k)=0$ and $\pi_{*} \mathscr{F}_{T}(k)$ is locally free of rank $n=P(k)$. On the other hand, $R^{1}\pi_{*}\mathscr{O}_{X_{T}}(1)=0$ and $\pi_{*}\mathscr{O}_{X_{T}}(1)$ is locally free of rank $N+1=10(2g-2)-g+1$. Let $\{U_{i}\}$ be an open cover of $T$ such that $(\pi_{*} \mathscr{F}_{T}(k))|_{U_{i}}\simeq \mathbb{C}^{n}\otimes \mathscr{O}_{U_{i}}$ and $(\pi_{*}\mathscr{O}_{X_{T}}(1))|_{U_{i}}\simeq \mathbb{C}^{N+1}\otimes \mathscr{O}_{U_{i}}$. This trivializations induce surjections
\begin{align}
\label{sur1}q_{2,W_i}:\mathbb{C}^{N+1}\otimes \mathscr{O}_{X_{T}}|_{W_{i}} & \twoheadrightarrow \mathscr{O}_{X_{T}}(1)|_{W_{i}},\\
\label{sur2}q_{1,W_i}:\mathbb{C}^{n}\otimes \mathscr{O}_{X_{T}}(-k)|_{W_{i}} & \twoheadrightarrow  \mathscr{F}_{T}|_{W_{i}},
\end{align}
where $W_{i}=\pi^{-1}(U_{i})$. Composing $\phi(ak)|_{W_{i}}$ with $(q_{1,W_i}(k)^{\otimes a})^{\oplus b}$ and taking the pushforward by $\pi$, we get a morphism
$$
\phi_{i}:((\mathbb{C}^{n})^{\otimes a})^{\oplus b}\otimes  \mathscr{O}_{T}|_{U_i}\rightarrow (\mathscr{N}|_{U_i})\otimes \pi_{*}\mathscr{O}_{X_{T}}(ak)|_{U_i}
$$
The first surjection, Equation (\ref{sur1}), embeds $q_{2,W_i}:W_i\hookrightarrow U_i \times \mathbf{P}^{M}$, while the second surjection, Equation (\ref{sur2}), defines a map $q_{1,W_i}:U_i \rightarrow \mathbf{Q}^{r}_{g}(\mu,k,P)$. Finally, the morphism $\phi_i$ defines a map $\phi_{i}:U_{i}\rightarrow \mathbf{P}((((\mathbb{C}^{n})^{\otimes a})^{\oplus b})^{\vee}\otimes\mu_{*}\mathscr{O}_{U_g}(ak))$ as well. Therefore, $q_{1,W_i}$ and $\phi_i$ define a map:
$$
\psi_i: U_i \rightarrow \mathbf{Q}^{r}_{g}(\mu,k,P)\times \mathbf{P}((((\mathbb{C}^{n})^{\otimes a})^{\oplus b})^{\vee}\otimes\mu_{*}\mathscr{O}_{U_g}(ak))
$$
Consider two open subsets $U_{i},U_{j}$, and let $U_{ij}=U_{i}\cap U_{j}$ be the intersection. Then, $\phi_i$ and $\phi_j$ define maps
$$
\xymatrix{
U_{ij} \ar@/^{5mm}/[rr]^{\phi_{i}} \ar@/_{5mm}/[rr]_{\phi_{j}} & & \mathbf{Q}^{r}_{g}(\mu,k,P)\times \mathbf{P}((((\mathbb{C}^{n})^{\otimes a})^{\oplus b})^{\vee}\otimes\mu_{*}\mathscr{O}_{U_g}(ak))
}
$$
which differ by a $U_{ij}$-valued point of the group scheme $\textrm{SL}_{n}\times\textrm{SL}_{M+1}$ and that take values in $\mathcal{Z}^{ss}_{[\alpha,\beta,\gamma]}$ because $\eta$ is $\delta$-semistable. Therefore, there is a well-defined morphism $T\rightarrow \mathcal{T}^{\delta\textrm{-(s)s}}_{P,g,a,b}$. This shows the existence of a natural transformation
$$
\Psi:\textrm{\textbf{Swamps}}^{\delta\textrm{-(s)s}}_{P,g,a,b}\longrightarrow \textrm{Hom}(-,\mathcal{T}^{\delta\textrm{-(s)s}}_{P,g,a,b}).
$$
Let $\mathcal{M}$ be  a scheme and suppose that there exists a natural transformation 
$$
\Psi':\textrm{\textbf{Swamps}}^{\delta\textrm{-(s)s}}_{P,g,a,b}\longrightarrow \textrm{Hom}(-,\mathcal{M})
$$
There is a canonical member $\phi_{univ}\in\textrm{\textbf{Swamps}}^{\delta\textrm{-(s)s}}_{P,g,a,b}(\mathcal{Z}^{\textrm{(s)s}}_{[\alpha,\beta,\gamma]})$ corresponding to the universal family (see Equation (\ref{unifamilyswamps})). The morphism 
$$\Psi'(\phi_{univ}): \mathcal{Z}^{\textrm{(s)s}}_{[\alpha,\beta,\gamma]}\longrightarrow\mathcal{M}$$ 
is $\textrm{SL}_{n}\times\textrm{SL}_{M+1}$-invariant, so it descends to a morphism 
$$\overline{\Psi'(\phi_{univ})}: \mathcal{T}^{\delta\textrm{-(s)s}}_{P,g,a,b}\longrightarrow \mathcal{M}$$
which defines a map $\Phi'':\textrm{Hom}(-,\mathcal{T}^{\delta\textrm{-(s)s}}_{P,g,a,b})\rightarrow \textrm{Hom}(-,\mathcal{M})$. Clearly, the triangle
$$
\xymatrix{
&\textrm{Hom}(-,\mathcal{T}^{\delta\textrm{-(s)s}}_{P,g,a,b})\ar[dd]^{\Psi''}\\
\textrm{\textbf{Swamps}}^{\delta\textrm{-(s)s}}_{P,g,a,b}\ar[ur]^{\Psi}\ar[dr]_{\Psi'} &\\
& \textrm{Hom}(-,\mathcal{M})
}
$$
commutes. Finally, the equality 
$
\textrm{\textbf{Swamps}}^{\delta\textrm{-(s)s}}_{P,g,a,b}(\mathbb{C})= \textrm{Hom}(\textrm{Spec}\mathbb{C},\mathcal{T}^{\delta\textrm{-(s)s}}_{P,g,a,b})
$ follows from Proposition \ref{closedpoints}.
\end{proof}

\section{A compactification of the universal moduli space of principal bundles}
\label{sectionmain}

Let $G$ be a semisimple linear algebraic algebraic group and $\rho:G\hookrightarrow \textrm{SL}(V)$ be a finite and faithful representation of dimension $r$.
Given a natural number $g\geq 2$, a rational number $\delta\in\mathbb{Q}_{>0}$ and a polynomial with integral coefficients of degree one, $P$, we will show that there is a projective coarse moduli space, $\textrm{SPB}(\rho)^{\delta\textrm{-(s)s}}_{P,g}$, for the moduli functor
\begin{equation*}
\textbf{SPB}(\rho)^{\delta\textrm{-(s)s}}_{P,g}(T)=\left\{
\begin{array}{l}
\textrm{isomorphism classes of pairs }(X_{T},( \mathscr{F}_{T},\tau_{T}))\\
\textrm{where }X_{T} \textrm{ is a semistable curve of genus }g \\
\textrm{over }T \textrm{ and } ( \mathscr{F}_{T},\tau_{T}) \textrm{ is a } \delta\textrm{-(semi)stable singular}\\
\textrm{principal }G\textrm{-bundle of uniform multi-rank }r \\
\textrm{over } T \textrm{ and with Hilbert polynomial } P
\end{array}\right\}.
\end{equation*}
together with a morphism  $\Theta_{sb}:\textrm{SPB}(\rho)^{\delta\textrm{-(s)s}}_{P,g}\rightarrow\overline{\textrm{M}}_{g}$. We will assume that each curve brings with it as polarization the very ample invertible sheaf given by $\mathscr{\omega}_{X}^{\otimes 10}$, which has degree $h=10(2g-2)$. In addition, we will show that if $P(n)=(rh)n+r(1-g)$ and $\delta$ is very large, then $\Theta_{sb}^{-1}([X])=\textrm{M}_{X}^{(s)s}(G)/\text{Aut}(X)$ for every smooth projective curve of genus $g$.

\subsection{The fiber-wise problem}

Let $X$ be a stable curve of genus $g$. A singular principal $G$-bundle over $X$ is a pair $(\mathscr{F},\tau)$ where $\mathscr{F}$ is a coherent sheaf of pure dimension one of rank $r$ and $\tau$ is a morphism of $\mathscr{O}_{X}$-algebras,
$\tau:S^{\bullet}(V\otimes \mathscr{F})^{G}\rightarrow\mathscr{O}_{X},$
which is not just the projection onto the zero degree component.

Consider a singular principal $G$-bundle on $X$,
$\tau:S^{\bullet}(V\otimes \mathscr{F})^{G}\rightarrow\mathscr{O}_{X}$.
We can fix $s\in\mathbb{N}$ such that $S^{\bullet}(V\otimes \mathscr{F})^{G}$ is generated by the submodule $\bigoplus_{i=0}^{s}S^{i}(V\otimes  \mathscr{F})^{G}$. Let $\underline{d}\in\mathbb{N}^{s}$ be such that $\sum id_{i}=s!$. Then we have:
\begin{equation}\label{asten}
\bigotimes_{i=1}^{s}(V\otimes \mathscr{F})^{\otimes id_{i}} \rightarrow\bigotimes_{i=1}^{s}S^{d_{i}}(S^{i}(V\otimes \mathscr{F}))\rightarrow \bigotimes_{i=1}^{s}S^{d_{i}}(S^{i}(V\otimes \mathscr{F}))^{G}\rightarrow\Ox
\end{equation}
Adding up these morphisms as $\underline{d}\in\mathbb{N}$ varies we find a swamp (see \cite{Alexander2})
\begin{equation}\label{associatedtensor}
\textrm{Swamp}(\tau):=\phi_{\tau}:((V\otimes \mathscr{F})^{\otimes s!})^{\oplus N}\rightarrow\Ox.
\end{equation}
Define $a:=s!$ and $b=N$. Then, by \cite{AMC}, there is an $s\in\mathbb{N}$ large enough such that the map
\begin{equation}\label{injective}
\left\{   \begin{array}{l}  
\textrm{isomorphism classes} \\ \textrm{of singular principal} \\ \textrm{G-bundles}  \end{array} \right\} \rightarrow 
\left\{   \begin{array}{l}  
\textrm{isomorphism classes} \\ \textrm{of swamps} \\ \textrm{of type }(a, b, \Ox)   \end{array} \right\} 
\end{equation}
is injective and depends only on the numerical input data, and not on the base curve.
 \begin{definition}\label{semistablebundle}
Let $\delta\in\mathbb{Q}_{>0}$. A singular principal $G$-bundle is said to be $\delta$-semi(stable) if its associated swamp is $\delta$-semi(stable)
 \end{definition}

Therefore, there is a natural transformation from the functor
$$
\textbf{SPB}(\rho)_{P,X}^{\delta\textrm{-(s)s}}(S)=\left\{   \begin{array}{l}\textrm{isomorphism classes of}\\ \textrm{families of } \delta\textrm{-(semi)stable} \textrm{ singular} \\ \textrm{principal G-bundles on }X \textrm{ parametrized }\\ \textrm{by S of uniform multi-rank } r\\ \textrm{and with Hilbert polynomial }P \end{array}  \right\}.
$$
to the functor 
\begin{equation*}
\textbf{Swamps}^{\delta\textrm{-(s)s}}_{P,\mathscr{O}_{X},a,b}(T)=\left\{
\begin{array}{l}
\textrm{isomorphism classes of }\delta\textrm{-(semi)stable swamps}\\
(V\otimes  \mathscr{F}_{T},\phi_{T},N)\textrm{ of uniform multi-rank }r \\
\textrm{and with Hilbert polynomial } P
\end{array}\right\}.
\end{equation*}

This allows to prove, following standard arguments (see \cite{AMC}), the next theorem.

 \begin{theorem}\label{B}\emph{\cite[Theorem 6.3]{AMC}}
There is a projective scheme $\emph{SPB}(\rho)_{P,X}^{\delta\text{-(s)s}}$ and an open subscheme $\emph{SPB}(\rho)_{P,X}^{\delta\text{-s}}\subset \textrm{\emph{SPB}}(\rho)_{P,X}^{\delta\text{-(s)s}}$ together with a natural transformation
\begin{equation*}
\alpha^{(s)s}\colon \textbf{\emph{SPB}}(\rho)_{P,X}^{\delta\textrm{-(s)s}}\rightarrow h_{\emph{SPB}(\rho)_{P,X}^{\delta\text{-(s)s}}}
\end{equation*}
with the following properties:

1) For every scheme $\N$ and every natural transformation $\alpha'\colon \emph{\textbf{SPB}}(\rho)_{P,X}^{\delta\textrm{-(s)s}}\rightarrow h_{\N}$, there exists a unique morphism $\varphi\colon \emph{SPB}(\rho)_{P,X}^{\delta\textrm{-(s)s}}\rightarrow\N$ with $\alpha'=h(\varphi)\circ\alpha^{(s)s}$.

2) The scheme $\emph{SPB}_{P,X}(\rho)^{\delta\textrm{-s}}$ is a coarse projective moduli space for the functor $\emph{\textbf{SPB}}_{P,X}(\rho)^{\delta\textrm{-s}}$
 \end{theorem}

\subsection{The relative problem}

Consider the Quot scheme
$
\mathbf{Q}^{r}_{g}(\mu,k,P)=\textrm{Quot}^{P,r}_{\mathbb{C}^{n}\otimes\mathscr{O}_{U_{g}}(-k)/U_{g}/ \mathbf{H}_g},
$
that was considered in the case of swamps, and the affine $\textrm{H}_{g}$-scheme defined by
\begin{equation*}
\mathbf{H}(V,s,k):=\bigoplus_{i=1}^{s} \textrm{\underline{Hom}}_{\textrm{H}_{g}}(S^{i}(V\otimes \mathbb{C}^{n})\otimes\mathscr{O}_{\textrm{H}_{g}},\mu_{*}\mathscr{O}_{U_{g}}(ik))\overset{\kappa}{\longrightarrow} \textrm{H}_{g}.
\end{equation*}
By \cite{AMC} and the results of Section \ref{sectionuniform}, we know that for $s,k$ large enough, every $\delta$-semistable singular principal $G$-bundle over a semistable curve of genus $g$ determines a point in $\mathbf{Q}^{r}_{g}(\mu,k,P)\times_{\textrm{H}_{g}}\mathbf{H}(V,s,k)$. Denote by $\overline{\kappa}$ and $\overline{\mu}$ the projections 
\begin{equation}
\begin{split}
(\mathbf{Q}^{r}_{g}(\mu,k,P)\times_{\textrm{H}_{g}}\mathbf{H}(V,s,k))\times_{\textrm{H}_{g}}U_{g}&\longrightarrow U_{g},\\
(\mathbf{Q}^{r}_{g}(\mu,k,P)\times_{\textrm{H}_{g}}\mathbf{H}(V,s,k))\times_{\textrm{H}_{g}}U_{g}&\longrightarrow \textrm{H}_{g},
\end{split}
\end{equation}
respectively. The goal now is to put a scheme structure on the locus given by the points $([q],[k])\in \mathbf{Q}^{r}_{g}(\mu,k,P)\times_{\textrm{H}_{g}}\mathbf{H}(V,s,k)$ that comes from a morphism of algebras
$
S^{\bullet}(V\otimes \mathscr{F})^{G}\rightarrow\Ox.
$
For the sake of clarity, let us denote $\mathbf{E}_{g}:=(\mathbf{Q}^{r}_{g}(\mu,k,P)\times_{\textrm{H}_{g}}\mathbf{H}(V,s,k))\times_{\textrm{H}_{g}}U_{g}$. On $\mathbf{E}_{g}$, there are universal morphisms
\begin{equation*}
\varphi'^{i}\colon S^{i}(V\otimes W)\otimes\mathscr{O}_{\mathbf{E}_{g}}\rightarrow \overline{\mu}^{*}\overline{\mu}_{*}(\overline{\kappa}^{*}\mathscr{O}_{U_g}(ik))
\end{equation*}
that composed with the evaluation maps $\overline{\mu}^{*}\overline{\mu}_{*}(\overline{\kappa}^{*}\mathscr{O}_{U_g}(ik))\twoheadrightarrow \overline{\kappa}^{*}\mathscr{O}_{U_g}(ik)$ lead to 
\begin{equation*}
S^{i}(V\otimes W)\otimes\mathscr{O}_{\mathbf{E}_{g}}\rightarrow \overline{\kappa}^{*}\mathscr{O}_{U_g}(ik)=\mathscr{O}_{\mathbf{E}_{g}}(ik)
\end{equation*}
Summing up these morphisms over all $i$ we find
$$\varphi_{\mathbf{E}_{g}}\colon \mathscr{V}_{\mathbf{E}_{g}}\colon =\bigoplus_{i=1}^{s}S^{i}(V\otimes W\otimes\mathscr{O}_{\mathbf{E}_{g}}(-k))\rightarrow\mathscr{O}_ {\mathbf{E}_{g}}.$$
Now $\varphi$ gives a morphism
$\tau'_{\mathbf{E}_{g}}\colon S^{\bullet}(\mathscr{V}_{\mathbf{E}_{g}})\rightarrow \mathscr{O}_{\mathbf{E}_{g}}.$
Consider now the universal quotient, $q_{U_g}:\mathbb{C}^{n}\otimes\phi^{*}\mathscr{O}_{U_{g}}(-k)\twoheadrightarrow \mathscr{F}$, on $\mathbf{Q}^{r}_{g}(\mu,k,P)\times_{\textrm{H}_{g}}U_{g}$. Pulling it back to $\mathbf{E}_{g}$ we get a quotient $q_{\mathbf{E}_{g}}:\mathbb{C}^{n}\otimes\mathscr{O}_{\mathbf{E}_{g}}(-k)\twoheadrightarrow\overline{ \mathscr{F}}$, and therefore a chain of surjections
$$
\xymatrix{
S^{\bullet}(V\otimes \mathbb{C}^{n}\otimes\mathscr{O}_{\mathbf{E}_{g}}(-k))\ar@{->>}[rr]^{S^{\bullet}(1\otimes q_{\mathbf{E}_{g}})} & & S^{\bullet}(V\otimes \overline{ \mathscr{F}})\ar@{->>}[d]^-{\textrm{Reynolds}} \\
S^{\bullet}\mathscr{V}_{\mathbf{E}_{g}}\ar@{->>}[u] & & S^{\bullet}(V\otimes \overline{ \mathscr{F}})^{G}
}
$$
Let us denote by $\beta$ the composition of these morphisms and consider the diagram
$$
\xymatrix{
0\ar[r] & \textrm{Ker}(\beta)\ar@{^(->}[r]\ar@{-->}[rd]^{\tau'_{\mathbf{E}_{g}}} & S^{\bullet}\mathscr{V}_{\mathbf{E}_{g}}\ar[r]^{\hspace{-1.2cm}\beta} \ar[d]^{\tau'_{\mathbf{E}_{g}}} & S^{\bullet}(V\otimes \overline{ \mathscr{F}})^{G} \ar[r] & 0 \\
 & & \mathscr{O}_{\mathbf{E}_{g}} & &
}
$$
As in the fiber-wise problem, there exists a closed subscheme 
$$\mathbf{D}_{g}\subset \mathbf{Q}^{r}_{g}(\mu,k,P)\times_{\textrm{H}_{g}}\mathbf{H}(V,s,k)$$ 
over which the morphism $\tau'_{\mathbf{E}_{g}}:S^{\bullet}\mathscr{V}_{\mathbf{E}_{g}}\rightarrow\mathscr{O}_{\mathbf{E}_{g}}$ lifts to a morphism of algebras
$S^{\bullet}(V\otimes \overline{ \mathscr{F}}_{|_{\mathbf{D}_{g}}})^{G}\rightarrow \mathscr{O}_{\mathbf{D}_{g}}$. There are two groups acting on $\mathbf{D}_{g}$. The action of the group $\textrm{SL}_{M+1}$ on $\textrm{H}_{g}$ lifts to an action on $\mathbf{D}_{g}$, while the group $\textrm{GL}_{n}$ is acting on both, $\mathbf{Q}^{r}_{g}(\mu,k,P)$ and $\mathbf{H}(V,s,k)$. The group $\textrm{GL}_{n}$ leaves invariant $\mathbf{D}_{g}$, so $\textrm{GL}_{n}$ is acting on $\mathbf{D}_{g}$ as well. Again, as in the fiber-wise problem, we can study the quotient by studying separately the actions of $\mathbb{C}^{*}$ and $\textrm{SL}_{n}$ (see \cite[Section 4.2]{Alexander2}).

Let $\mathcal{Z}$ be the parameter space for swamps $\phi:((V\otimes \mathscr{F})^{\otimes a})^{\oplus b}\rightarrow \mathscr{O}_{X}$ (see Subsection \ref{parameterswamps}). The injective map defined in Equation (\ref{injective}) shows that there is a well-defined morphism of $\textrm{H}_{g}$-schemes
$$
\textrm{Swamp}: \mathbf{D}_{g}\longrightarrow \mathcal{Z}
$$
which is $\mathbb{C}^{*}$-invariant and $\textrm{SL}_{n}$-equivariant, injective and proper. Thus, it induces a morphism
$$
\overline{\textrm{Swamp}}: \overline{\mathbf{D}_{g}}:=\mathbf{D}_{g}\slas\mathbb{C}^{*}\longrightarrow \mathcal{Z}
$$
which is $\textrm{SL}_{n}$-equivariant, injective and proper. Finally, by Definition \ref{semistablebundle} and Theorem \ref{existenceswamps} we conclude that $\textrm{SPB}(\rho)^{\delta\textrm{-(s)s}}_{P,g}:=\overline{\mathbf{D}_{g}}\slas(\textrm{SL}_{M+1}\times\textrm{SL}_{n})$ exists and is projective.
 \begin{theorem}\label{Mainteo2}
The projective scheme  $\emph{\textrm{SPB}}(\rho)^{\delta\textrm{-(s)s}}_{P,g}$ is a coarse moduli space for the moduli functor  $\emph{\textbf{SPB}}(\rho)^{\delta\textrm{-(s)s}}_{P,g}$.
 \end{theorem}
\begin{proof}
Follows from standard arguments as those given in \cite[Proposition 4.1.1, Proposition 4.2.1, Theorem 4.2.2.]{Alexander2}, and from Theorem \ref{Mainteo1}.
\end{proof}

\subsection{The fibers over the nonsingular locus $\textrm{M}_{g}$}\label{fibers}

Assume now that $P(n)=(rh)n+r(1-g)$. This is the Hilbert polynomial of coherent sheaves of uniform multi-rank $r$ and degree $0$. We have constructed a projective scheme $\textrm{SPB}(\rho)^{\delta\textrm{-(s)s}}_{P,g}$ together with a map $\Theta_{sb}:\textrm{SPB}(\rho)^{\delta\textrm{-(s)s}}_{P,g}\rightarrow\overline{\textrm{M}}_{g}$ satisfying that for any stable curve $[X]\in\overline{\textrm{M}}_{g}$, 
$$\Theta_{sb}^{-1}([X])=\textrm{SPB}(\rho)^{\delta\textrm{-(s)s}}_{P,X}/\textrm{Aut}(X).$$
By \cite[Theorem 3.3.1]{Alexander3}, there exists $\delta_{\infty}\in\mathbb{Q}_{>>0}$ large enough, depending only on $a,b,P$, such that the associated swamp of any $\delta$-semistable singular principal $G$-bundle on a smooth projective curve of genus $g$, $( \mathscr{E},\tau)\in\textrm{SPB}(\rho)^{\delta\textrm{-(s)s}}_{P,X}$, is generically semistable, and by \cite[Corollary 4.1.2]{Alexander1} and \cite[Remark 2.3.4.4]{Asch}, this means that $( \mathscr{E},\tau)$ is honest and $ \mathscr{E}$ is a semistable vector bundle. Finally, form \cite[Paragraph 5.1]{Alexander2}, we deduce that $\textrm{SPB}(\rho)^{\delta\textrm{-(s)s}}_{P,X}=\textrm{M}_{X}(G)$ for any $\delta>\delta_{\infty}$ and any smooth projective curve of genus $g$. Therefore, if we assume $\delta>\delta_{\infty}$, we have
$$\Theta_{sb}^{-1}([X])=\textrm{M}_{X}(G)/\textrm{Aut}(X)$$ 
for every smooth projective curve of genus $g$, thus, $\textrm{SPB}(\rho)^{\delta\textrm{-(s)s}}_{P,g}$ is a compactification of the moduli problem defined by pairs $(X,P)$ where $X$ is a smooth projective curve of genus $g$ and $P$ is a principal $G$-bundle. 

Observe that a good compactification, in the sense of Pandharipande, requires for $\Theta_{sb}^{-1}(\textrm{M}_{g})$ to be a dense open subset of $\textrm{SPB}(\rho)^{\delta\textrm{-(s)s}}_{P,g}$. Since it is open, what remains to be solved is the part regarding the density. Note also that this is equivalent to the following statement: Let $G$ be a semisimple linear algebraic group, $\rho:G\hookrightarrow\textrm{SL}(V)$ a faithful representation of dimension $r$, $\delta\in\mathbb{Q}_{>0}$ larger than $\delta_{\infty}$ and $g\geq 2$ a natural number. Let $X$ be a stable curve of genus $g$ and $( \mathscr{F},\tau)$ a $\delta$-semistable singular principal $G$-bundle of degree $0$ and rank $r$. Then, there is a discrete valuation ring $(\Oo,\mathfrak{m},\mathbb{C})$, a flat family of stable curves of genus $g$, $X_{T}\rightarrow T:=\textrm{Spec}(\Oo)$, the generic fiber $X_{T,\mu}$ is smooth and $X_{T,(0)}\simeq X$, and a flat family of $\delta$-semistable singular principal $G$-bundles $( \mathscr{F}_{T},\tau_{T})$ such that $( \mathscr{F}_{T,(0)},\tau_{T,(0)})\simeq( \mathscr{F},\tau)$. By \cite[Lemma 9.2.3]{Pando}, the extension of the sheaf $ \mathscr{F}$ is ensured, so the problem we have to deal with is the extension of the morphism of algebras $\tau:S^{\bullet}(V\otimes \mathscr{F})^{G}\rightarrow\Oo_{X}$.

%
%


\begin{thebibliography}{99}

\bibitem{Balaji}    V. Balaji,  `Torsors on semistable curves and degenerations', arXiv:1901.01529v4 (2020).

\bibitem{bosle}    U. Bhosle,  `Tensor fields and singular principal bundles', \emph{International Mathematics Research Notices}, 2004
(2004) 3057--3077.

\bibitem{Fal}   G. Faltings,  `Moduli stacks for bundles on semistable curves', \emph{Mathematische Annalen}, 304 (1996) 489--515.

\bibitem{gies-curves}   D.  Gieseker,  \emph{Lectures on moduli of curves}. (Springer-Verlag, Heidelberg, 1982). 

\bibitem{sols}   T. G\'omez,  \and   I. Sols, `Stable tensors and moduli space of orthogonal sheaves', ArXiv: math.AG/0103150 (2001).

\bibitem{sols2}   T. G\'omez,  \and   I. Sols, `Moduli space of principal sheaves over projective varieties', \emph{Annals of Mathematics},  161 (2005) 1037--1092.

\bibitem{groth0}   A. Grothendieck,  `Sur la classification des fibres holomorphes sur la sphere de Riemann', \emph{American Journal of Mathematics},  79 (1957) 121--138.

\bibitem{FGA}   A. Grothendieck,  \emph{Fondements de la g\'eométrie alg\'ebrique [Extraits du Séminaire Bourbaki, 1957--1962]}.

\bibitem{Huyb}   D. Huybrechts,  \and M. Lehn,  \emph{The geometry of moduli spaces of sheaves (second edition)}. (Cambridge University Press, Cambridge, 2010).

\bibitem{L}    A.  Langer, `Moduli spaces of principal bundles on singular varieties', \emph{Kyoto Journal of Mathematics},  53 (2013) 3--23.

\bibitem{matsumura}    H. Matsumura,  \emph{Commutative ring theory}. (Cambridge University Press, Cambridge, 1986).

\bibitem{AMCdis} A. L. Mu\~noz Casta\~neda, Principal G-bundles on nodal curves, PhD thesis, Berlin, 2017,
ii+158 pp. Available at \url{http://www.diss.fu-berlin.de/diss/content/below/index.xml}.

\bibitem{AMC}    A. L. Mu\~noz Casta\~neda,  `On the moduli spaces of singular principal bundles on stable curves', (to appear in Adv. Geometry) arXiv:1806.09741  (2018).

\bibitem{mumford0}    D. Mumford,  Projective invariants of projective structures and applications, \emph{Proceedings of the International Congress of Mathematicians} (1962) 526--530.

\bibitem{mum-red}   D. Mumford,  \emph{Red book of varieties and schemes}, (Springer-Verlag, Berlin, 1999).

\bibitem{Pando}    P.  Pandharipande, A compactification over Mg of the universal moduli space of slope-semiestable vector bundles, \emph{Journal of the American Mathematical Society},  9 (1996) 425--471.

\bibitem{Rama}    A. Ramanathan, `Moduli of principal bundles over algebraic curves I.', \emph{Proceedings of the Indian Academy of Science},  106 (1996) 301--328.

\bibitem{Rama2}     A. Ramanathan, `Moduli of principal bundles over algebraic curves II.', \emph{Proceedings of the Indian Academy of Science},  106 (1996) 421--449.

\bibitem{Alexander2}    A.  Schmitt, `Singular principal bundles over higher-dimensional manifolds and their moduli spaces', \emph{International Mathematics Research Notices},  2002 (2002) 1183--1210.

\bibitem{Alexander3}    A.  Schmitt, `Global boundedness for decorated sheaves', \emph{International Mathematics Research Notices},  68 (2004) 3637--3671.

\bibitem{Alexander4}    A.  Schmitt,  `Moduli spaces for semistable honest singular principal G-bundles on a nodal curve which are compatible with degenerations-A remark on Bhosle's paper Tensor fields and singular principal bundles''', \emph{International Mathematics Research Notices},  2005 (2005) 1427--1436.

\bibitem{Alexander1}    A.  Schmitt, `Singular principal G-bundles on nodal curves', \emph{Journal of the European Mathematical Society},  7 (2005) 215--251.

\bibitem{Asch}    A.  Schmitt, \emph{Geometric invariant theory and decorated principal bundles}, EMS Zurich Lectures in Advances Mathematics, Zurich , 2008.

\bibitem{serre} J. P. Serre, `Espaces fibr\'es alg\'ebriques', Séminaire Claude Chevalley, Tome 3, Exposé no. 1, 37 p. (1958)

\bibitem{Seshadri2}    C. S. Seshadri, `Space of unitary vector bundles on a compact Riemann surface', \emph{Annals of Mathematics},  85 (1967) 303--336.

\bibitem{sidman}    J. Siedman, `On the Castelnuovo-Mumford regularity of products of ideal sheaves', \emph{Advances in Geometry},  2 (2002) 219--229.

\bibitem{simpson}     C.  Simpson,  `Moduli of representations of the fundamental group of a smooth projective variety', \emph{Publications Math\'ematiques de l'IHES},  79 (1994) 47--129.

\bibitem{Solis}    P. Sol\'is,  `A complete degeneration of moduli of $G$-bundles on a curve', arXiv:1311.6847v2 (2013).

\bibitem{sun1}    X.  Sun, `Degenerations of SL(r)-bundles on a reducible curve', \emph{Proceedings of the Symposium on Algebraic Geometry in East Asia}, (2001) 3--10.

\bibitem{sun2}    X.  Sun, Moduli spaces of SL(r)-bundles on singular irreducible curves', \emph{Asian Journal of Mathematics},  7 (2003) 609--626.

%
%
%
%
%
%
%
%
%
%
%
%
%
%
%
%
%
%
%
%
%
%
%
%
%
%
%
%
%
%




\end{thebibliography}
\end{document}